\documentclass[12pt]{article}

\usepackage{lipsum}
\usepackage{fancyhdr}
\usepackage{booktabs}
\usepackage{array}
\usepackage{amssymb}
\usepackage{parskip}
\usepackage{hyperref}
\usepackage{enumitem}
\usepackage{mathtools} 
\usepackage{xparse,textcomp}
\DeclarePairedDelimiter\floor{\lfloor}{\rfloor}

\usepackage[lofdepth,lotdepth]{subfig}

\usepackage{cite}

\usepackage{parskip}
\usepackage[margin=0.8in]{geometry}
\newcommand{\forceindent}{\leavevmode{\parindent=1.5em\indent}}
\usepackage{algorithm} 
\usepackage{algpseudocode}
\newcommand{\R}{\mathbb R}

\newcommand{\maxmath}{\mathop{}\!\mathrm{max}}
\newcommand{\minmath}{\mathop{}\!\mathrm{min}}
\newcommand{\limmath}{\mathop{}\!\mathrm{lim}}
\newcommand{\Prmath}{\mathop{}\!\mathrm{Pr}}
\usepackage{amsmath,amsthm}
\makeatletter

\usepackage[belowskip=-30pt,aboveskip=10pt]{caption} 
\usepackage{amsfonts}
\usepackage{amsopn}

\usepackage{amssymb,amsmath}

\renewcommand\section{\leftskip 0pt\@startsection {section}{1}{\z@}%
	{-3.5ex \@plus -1ex \@minus -.2ex}%
	{2.3ex \@plus.2ex}%
	{\normalfont\Large\bfseries}}

\renewcommand\subsection{\leftskip 0pt\@startsection{subsection}{2}{\z@}%
	{-3.25ex\@plus -1ex \@minus -.2ex}%
	{1.5ex \@plus.2ex}%
	{\normalfont\large\bfseries}}

\renewcommand\subsubsection{\leftskip 0pt\@startsection{subsubsection}{3}{\z@}%
	{-3.25ex\@plus -1ex \@minus -.2ex}%
	{1.5ex \@plus .2ex}%
	{\normalfont\large\bfseries}}
\usepackage{graphicx}

\usepackage{geometry}
\newtheorem{definition}{Definition}

\newtheorem{Theorem}{Theorem}[section]

\newtheorem{corollary}[Theorem]{Corollary}

\theoremstyle{remark}
\newtheorem{remark}[Theorem]{Remark}

\begin{document}
	
\title{ \centering{Manifold Reconstruction and Denoising from Scattered Data in High Dimension via a Generalization of $L_1$-Median}}
\author{Shira Faigenbaum-Golovin${^{1, *}}$~~David Levin${^1}$ \\
	\small{${^1}$ School of Mathematical Sciences, Tel Aviv University, Israel}
	\\
	\small{${^*}$ Corresponding author, E-mail address: alecsan1@post.tau.ac.il} 
}


\maketitle
\begin{abstract}
In this paper, we present a method for denoising and reconstruction of low-dimensional manifold in high-dimensional space. We suggest a multidimensional extension of the Locally Optimal Projection algorithm which was introduced by Lipman et al. in 2007 for surface reconstruction in 3D. The method bypasses the curse of dimensionality and avoids the need for carrying out dimensional reduction. It is based on a non-convex optimization problem, which leverages a generalization of the outlier robust L1-median to higher dimensions while generating noise-free quasi-uniformly distributed points reconstructing the unknown low-dimensional manifold. We develop a new algorithm and prove that it converges to a local stationary solution with a bounded linear rate of convergence in case the starting point is close enough to the local minimum. In addition, we show that its approximation order is $O(h^2)$, where $h$ is the representative distance between the given points. We demonstrate the effectiveness of our approach by considering different manifold topologies with various amounts of noise, including a case of a manifold of different co-dimensions at different locations.

\end{abstract}

\noindent\textbf{keywords:} Manifold learning, Manifold denoising, Manifold reconstruction, High dimensions, Dimensional reduction

(Numerical analysis - Numerical approximation and computational geometry)

\section{Introduction}
High-dimensional data is increasingly available in many fields, and the problem of extracting valuable information from such data is of primal interest. Often, the data suffers from the presence of noise, outliers, and non-uniform sampling, which can influence the result of the mining task.  We can address this problem by denoising a single sample, an approach extensively used in the last decades (the denoising method is often data-driven). However, it is still a challenge to produce a good noise-free result from a single sample with a large amount of noise present. Frequently, classical denoising algorithms lose the battle, since they denoise a single sample and overlook the intrinsic connections between different samples acquired from a chosen domain. As a result, obtaining a dataset of samples with certain properties can boost the denoising process. A common practice is to assume that the high-dimensional input data lies on an intrinsically low-dimensional Riemannian manifold.

\forceindent For instance, with the development of image processing, the task of image denoising gained a lot of attention (see, e.g., \cite{elad2006image, starck2002curvelet, mahmoudi2005fast}). Thus, given a single image, the task is to find its noise-free image. Now, let us consider a collection of noisy images depicting a single object, controlled by several parameters (such as a set of faces or written letters rotated in different directions). This collection can be modeled by a manifold, and this representation can be utilized to produce a superior denoising result. A real-life case, which motivated the current research, is cryo-electron microscopy \cite{singer2011viewing}. In this problem a single image is a projection of a three-dimensional macromolecule into a two-dimensional representation (Figure \ref{fig:cryoFig} (A)). Cryo-electron microscopy images are known to suffer from extremely low signal to noise ratio (Figure \ref{fig:cryoFig} (C)), and consequently classical denoising methods usually do not perform well on such samples. Nevertheless, using the fact that the images are sampled from a manifold (each corresponding to the molecule projected in a different direction) can facilitate the denoising task. Figure \ref{fig:cryoFig} (B) shows a collection of images, each depicting a projection of the simulated molecule in Figure \ref{fig:cryoFig} (A), captured in various directions. Thus, we transfer the problem from single image denoising to denoising the entire image set – which is treated as scattered data sampled from a manifold. 

\begin{figure}[H] 
	\centering
	\label{fig:b}\includegraphics[width=\textwidth,height=\textheight,keepaspectratio]{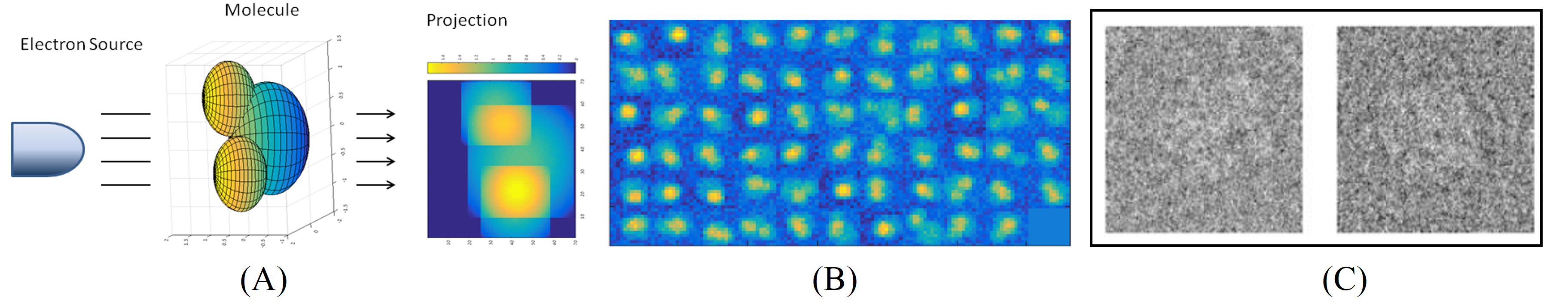}
	\caption{(A) Illustration of the cryo-electron microscopy projection process, where a 3D molecule is rotated and projected to 2D. (B) Collection of the artificial projections of the molecule with noise, where each image is the molecule rotated in a different direction. (C) Two real electron microscope images of the E. coli 50S ribosomal subunit (image is taken from \cite{singer2011viewing}). These images demonstrate the denoising challenge of extremely low signal to noise ratio.}
	\label{fig:cryoFig}
\end{figure}

In this paper, we address the problem of manifold denoising and reconstruction. Let $\mathcal{M}$ be a $d$-dimensional manifold in $\mathbb{R}^n$, where $d\ll n$. Suppose that the scattered data $P=\{p_j\}_{j=1}^J$ were sampled near $\mathcal{M}$ and contain noise and outliers. We wish to find a noise-free reconstruction of the geometry of $\mathcal{M}$ in $\mathbb{R}^n$.

\forceindent Before we turn to high-dimensional data, we first consider the simpler, yet challenging problem of surface reconstruction. While the problem of low-dimensional reconstruction was thoroughly studied along the years \cite{alexa2003computing, berger2017survey, cohen1999progressive, levin2004mesh, lipman2007parameterization}, there are still many challenges which modern applications (e.g., computer graphics) pose. One of them is surface reconstruction with preservation of features \cite{huang2013edge, yadav2018constraint}. The available methods commonly assume almost noise-free data and rely on normal estimation. Unfortunately, in real-life cases, noise is often present, and normal estimation may not be robust enough (despite various processes for cleaning the normals). Let us mention here the Parameterization-free Projection method for geometry reconstruction proposed in \cite{lipman2007parameterization}, which offers a solution that can handle high levels of noise. This method does not require a well-defined surface parameterization, avoids using local surface approximation and normal estimation, is cheap, and can be parallelized due to its local support. In \cite{lipman2007parameterization}, it was demonstrated by various examples that the method is stable with respect to outliers, different density of sampling and varying topology. 

\forceindent In the high-dimensional case, the problem of manifold reconstruction still requires additional attention. The era of proliferation of high-dimensional data raised the need for efficient denoising and reconstruction algorithms for manifolds. The application of classical approximation tools, developed for surfaces, to high-dimensional data, encounters various challenges, usually stemming from the high-dimension, and presence of noise. For instance, given a uniform sampling  in $\R^{n}$ on a grid with spacing $h=1/L$ requires $L^n$ samples and when $L \ll 10$ this is already challenging for $n \ll 10$. Moreover, classical approximation methods assume smoothness of order $s$, which is closely related to the approximation error. For example, for $J$ sample points, the reconstruction accuracy can be of the order of $O(J^{-s/n})$, which implies that we need to increase the amount of data as the domain dimension increases \cite{bachmayr2014approximation}. As a result, in the high-dimensional case, the problem of manifold reconstruction still requires additional attention especially to the problem of denoising and reconstructing manifold.

\forceindent A common way of dealing with high-dimensional data is to use dimensionality reduction. The motivation often stems from the need to analyze, process, and visualize high-dimensional data. Along the years many dimensionality reduction techniques were developed (PCA \cite{pearson1901liii}, Multidimensional Scaling \cite{cox2000multidimensional}, Linear Discriminant Analysis \cite{fisher1936use}, Locality Preserving Projections \cite{he2004locality},  Locally Linear Embedding \cite{roweis2000nonlinear}, ISOMAP \cite{tenenbaum2000global}, Diffusion Maps \cite{coifman2005geometric}, and Neural Networks in their general form, \cite{lin2008riemannian}, to mention just a few). However, one has to be careful when performing dimensionality reduction, since meaningful information can be lost due to the assumptions made. One fundamental challenge of dimensionality reduction is knowing or estimating the dimension of the data. In addition, since the geometry of the data is usually unknown, it is common to use an assumption regarding its geometrical structure (and use linear/non-linear algorithms accordingly). As a result, in the case of real-life data, it is still a challenge to address these issues, mainly because such assumptions have a direct influence on the usage of dimensionality reduction methods, and may, therefore, hamper the results of the analysis performed. For a comprehensive survey of manifold learning methods that rely on dimensionality reduction, see \cite{lin2008riemannian}.

\forceindent An alternative practice for handling high-dimensional data is manifold learning in high-dimensional space. Thus, instead of making assumptions on the geometry of the manifold, its intrinsic dimension and reducing the dimension of the data, the mining task is performed in a high-dimensional space. This approach has several advantages, as well as disadvantages. On the one hand, there is no loss of information. On the other hand, the dimension of the data influences the efficiency and feasibility of the algorithms, and it is possible that one will not be able to see the forest for the tree. An additional important factor of high-dimensional data is noise, which is usually present in real-life scenarios. In Table \ref{survey_table} we give a short survey of manifold reconstruction methods that avoid performing dimensionality reduction. Among the first papers that addressed the manifold reconstruction problem was \cite{cheng2005manifold}. The method presented therein relied on Delaunay triangulation, and as the authors themselves noted, it was impractical, mainly because it requires a very dense and noise-free sample, and also because it makes use of (weighted) Delaunay triangulation in higher dimensions. Next, in \cite{niyogi2008finding2} it was proposed to use simplicial complexes. In that paper, the authors also address the challenge of noisy samples, under certain conditions. This work was followed by \cite{boissonnat2009manifold}, which aimed at avoiding computing the Delaunay triangulation of the given set of points by using a Witness complex via an iterative process, and by \cite{boissonnat2014manifold} which addressed the problem using a Tangential Delaunay complex. Unfortunately, this method dealt only with noise-free samples. Next in \cite{maggioni2016multiscale}, the authors proposed to learn a data-dependent dictionary from clean data in the chosen resolution level and use it for the manifold reconstruction of possibly noisy data. Later, in \cite{fefferman2018fitting}, it was suggested to use a covering of the manifold by discs to deal with a small amount of Gaussian noise. The recent paper \cite{sober2016manifold} proposed to address manifold denoising under various noisy scenarios, when the intrinsic dimension of the manifold is known, by extending the Moving Least Squares method \cite{levin2004mesh} to the high-dimensional case. Finally, the paper \cite{aamari2019nonasymptotic} address the problem manifold reconstruction and of tangent space and curvature estimation by using local polynomials.

\newcolumntype{P}[1]{>{\raggedright\arraybackslash}p{#1}}
\setlength{\extrarowheight}{8pt}
\footnotesize

\begingroup
\setlength{\tabcolsep}{8pt} 
\renewcommand{\arraystretch}{1} 

\begin{table}[H]
	\caption{Survey of manifold reconstruction methods, that avoid performing dimensionality reduction}
	\footnotesize
	\linespread{0.8}\selectfont\centering
	\label{survey_table}
	\begin{tabular}{@{}P{1.5cm}P{1.4cm}P{2.1cm}P{2.3cm}P{1.3cm}P{1.1cm}P{1.4cm}P{2.9cm}@{}} 
		\toprule
		Authors                                          & Algorithm key features                       & Assumptions on the data                                                & Sampling                                                                                     & Handle noise                               & Error   & Numerical exp. & Complexity, $N$ is $\#$points, $d$ is ID, $n$ is the dim. of the ambient space                                                                                                                                                                    \\ \midrule
		Cheng, et al. \cite{cheng2005manifold}           & Weighted Delaunay triangulation & Compact manifold, smooth, no boundary                   & Sufficiently dense point sample                                                            & Noise-free sample                          & homeo-morphic & N/A                   & $O(N\text{ log}(N))$                                                                                                                                                                                
		\\
		Niyogi, et al. \cite{niyogi2008finding2}         & Simplicial complex                &   Suffcient amount of  points                                                                     &                                               & Bounded/ specific models of noise & homeo-morphic & N/A                   & N/A                                                                                                                                                                                          \\
		Boissonnat, et al. \cite{boissonnat2009manifold} & Witness complex                   & Positive reach (i.e. $C^1$-continuous)                                   & Not necessarily uniformly sampled, minimal local density & Low noise level                            & homeo-morphic & N/A                   & $N^2 d^{O(d^2)}$                                                                                                                                                                                
		
		\\
		Chazal, et al. \cite{chazal2011geometric}            & Distance functions with probability distribution                              &                                        &    Regularity of the input data                                                                     & Bounded/ specific models of	noise                             & homotopic            & \checkmark                     & N/A
		
		\\
		Boissonnat, et al. \cite{boissonnat2014manifold} & Tangential Delaunay complex      &  Smooth manifold, positive reach                                
		& Sampling ratio, point sparsity, and the reach hold a condition  & Noise-free sample                            & homeo-morphic & N/A                   & $O(n)N^2+n2^{O(d^2)}N$                                                                                                     \\
		Maggioni, et al. \cite{maggioni2016multiscale}  & Dictionary                        & Smooth closed manifold, $d$ is known  & Homogeneous, reconstruct new noisy samples    & Additive noise, dictionary is built from clean samples                        & \checkmark            & \checkmark                     & $O(C^d (n+d^2)\epsilon^{-(1-\frac{d}{2})} log \frac{1}{\epsilon}+dn)$, where $C$ is a constant, and  $\epsilon$ is reconstruction error                                                                                                                                                                                                                                                                                                                                                       \\
		Fefferman, et al. \cite{fefferman2018fitting}   & Disk stitching                    & Reach is bounded                                                       &                                                                                              & Additive noise                             & \checkmark            & N/A                   & N/A                                                                                                                                                                                          \\
		Sober, Levin \cite{sober2016manifold}            & Moving Least Squares                               & $d$ is known, bounded reach                                       &                                                                                              & Additive noise                             & \checkmark            & \checkmark                     & $O(d^3m+Nd^m+NI)$, $I$-$\#$points in supp., $m$ is the approx. degree
		\\
		Aamari, Levrard \cite{aamari2019nonasymptotic}            & Local Polynomials                               & $d$ and order of regularity are known                                       &                                                                                              & Bounded/ specific models of	noise                             & \checkmark            & N/A                     & N/A

		\\ \bottomrule
	\end{tabular}
	\vspace{-10mm}
	
\end{table}
\endgroup
\normalsize

The methods listed in the table provide a strong theoretical background, but most of them are not accompanied by numerical examples (except \cite{chazal2011geometric, maggioni2016multiscale, sober2016manifold}), which is an important aspect of evaluating the method execution. In addition, unfortunately, as can be seen from the table, handling noisy data, non-uniformly sampled, with no assumption on the data, is still a challenge in high-dimensional cases. In this paper, we propose denoising and reconstructing the manifold geometry in a high-dimensional space in the presence of high amounts of noise and outliers. We will tackle the manifold approximation question by extending the Locally Optimal Projection algorithm \cite{lipman2007parameterization} to the high-dimensional case. The proposed algorithm is simple, fast and efficient, and does not require any additional assumption. Our theoretical analysis is accompanied by numerical examples of various manifolds with different amounts of noise.

\section{High-Dimensional Denoising and Reconstruction}

\label{sec:framework}

The Locally Optimal Projection (LOP) method was introduced in \cite{lipman2007parameterization} to approximate two-dimensional surfaces in $\mathbb{R}^3$ from point set data.The procedure does not require the estimation of local normals and planes, or parametric representations. In addition, the method performs well in the case of noisy samples. Due to its flexibility and satisfactory results, it has been extended to address other challenges related to surfaces \cite{huang2013edge, huang2009consolidation, su2011curvature}. 

\forceindent Herein we generalize the LOP mechanism to perform what we call \textit{Manifold Locally Optimal Projection (MLOP)}. The vanilla LOP is not able to cope with high-dimensional data, mainly due to the sensitivity of the norm to noise and outliers (as will be discussed in details in subsection \ref{sec:highDimDistSec}). In addition, other adaptations are required due to practical reasons (as will be described in the end of this section). 

\forceindent First, we adapt the $h$-$\rho$ condition defined for scattered-data approximation functions (in \cite{levin1998approximation}, defined for low-dimensional data), to handle finite discrete data on manifolds.

\begin{definition}\label{def:def4}
	\textbf{$h$-$\rho$ sets of fill-distance $h$, and density $\leq \rho$} with respect to the manifold $\mathcal{M}$. Let $\mathcal{M}$ be a $d$-dimensional manifold in $\mathbb{R}^n$ and consider a set of data points  $P=\{{p_j }\}_{i=1}^J$ sampled from $\mathcal{M}$. We say that $P$ is an 
	$h$-$\rho$ set if:\\
	1.   $h_0$ is the fill-distance, i.e., $h_0=\maxmath_{y\in M}\minmath_{p_j\in P}\|y-p_j\|$.\\
	2.   The density of the points can be bounded as $\#\{P \cap \bar{B}(y,kh_0)\}\leq \rho k^d, \quad k \geq 1, \quad y\in \mathcal{M}$. \\
	Here $\#Y$ denotes the number of elements in a set $Y$ and $\bar{B}(x,r)$ denotes the closed ball of radius $r$ centered at $x$. 
\end{definition}

\forceindent Note that the last condition regarding the point separation $\delta$ defined in \cite{levin1998approximation}, which states that there $\exists \delta >0$ such that $\|p_i-p_j \| \geq \delta, \quad 1\leq i \leq j \leq J$, is redundant in the case of finite data. 

\forceindent The setting for the high-dimensional reconstruction problem is the following: Let $\mathcal{M}$ be a manifold in $\mathbb{R}^n$, of unknown intrinsic dimension $d \ll n$. One is given a noisy point-cloud $P=\{p_j\}_{j=1}^J \subset \mathbb{R}^n$ situated near the manifold $\mathcal{M}$, such that $P$ is a $h$-$\rho$ set. We wish to find a new point-set $Q=\{q_i\}_{i=1}^I \subset \mathbb{R}^n$  which will serve as a noise-free approximation of $\mathcal{M}$. We seek a solution in the form of a new point-set $Q$, which will replace the given data $P$, provide a noise-free approximation of $\mathcal{M}$, and which is quasi-uniformly distributed. This is achieved by leveraging the well-studied weighted $L_1$-median \cite{vardi2000multivariate} used in the LOP algorithm and requiring a quasi-uniform distribution of points $q_i \in Q$. 
These ideas are encoded by the cost function

\vspace{-12mm}
\begin{equation} \label{eq:1}	 G(Q) = E_1(P,Q)+\Lambda E_2(Q) = \sum\limits_{q_i \in Q} \sum\limits_{pj \in P} \|q_i-p_j\|_{H_{\epsilon}} w_{i,j} +
\sum\limits_{q_i \in Q} \lambda_i \sum\limits_{q_{i'} \in Q \backslash \{q_i\}} \eta(\|q_i-q_i'\|) \hat w_{i,i'}\,, \end{equation}
where the weights $w_{i,j}$ are given by rapidly decreasing smooth functions. In our implementation we used $w_{i,j} = \exp\big\{\!-\|q_i-p_j\|^2/{h_1^2}\big\}$
and $\widehat w_{i,i'} = \exp\big\{\!-\|q_i-q_i'\|^2/{h_2^2}\big\}$. Here, we replace the $L_1$-norm used in \cite{lipman2007parameterization} by the "norm" $\| \cdot \|_{H_{\epsilon}}$ introduced in \cite{levin2015between} as $\|v\|_{H_\epsilon}=\sqrt{v^2+\epsilon}$, where $\epsilon >0$ is a fixed parameter (in our case we take $\epsilon=0.1$). As shown in \cite{levin2015between}, using $\| \cdot \|_{H_{\epsilon}}$ instead of $\| \cdot \|_1$ has the advantage that one works with a smooth cost function and outliers can be removed. In addition, $h_1$ and $h_2$ are the support size parameters of $w_{i,j}$ and $\hat w_{i,i'}$ that guarantee a sufficient amount of $P$ or $Q$ points for the reconstruction. We provide additional details on how to estimate the support size, in Subsection \ref{sec:optH}. Also, $\eta(r)$ is a decreasing function such that $\eta(0)= \infty $; in our case we take $\eta(r) = \frac {1} {3r^3}$. Finally, $\{\lambda_i\}_{i=1}^I$ are constant balancing parameters.

\forceindent We will now give some intuition about the definition of the cost function $G$. We can describe the cost function in \eqref{eq:1} in terms borrowed from electromagnetism, where an electron generates an electric field that exerts an attractive force on a particle with a positive charge, such as the proton, and a repulsive force on a particle with a negative charge. In our scenario, we have attraction forces between the $Q$-points and the original $P$-points, and repulsion forces between the $Q$-points to themselves in order to make them spread out in a quasi-uniform manner (Figure \ref{fig:attractionForces}). An additional way of looking at the target function is to view the solution using a service center approach: placing a distribution of service centers $q_i \in Q$ to best serve the customers $P$, such that the service centers are spread uniformly. Thus, in case we have more points in $P$ than in the reconstruction, each center $q_i \in Q$ will serve a certain amount of $P$-points in its neighborhood.
\begin{remark}
	We do not require that the amount of the points in the reconstruction ($Q$), and the size of the original sample set ($P$) be the same. This flexibility allows downsampling and upsampling in order to decode or encode manifold information.
\end{remark}
\vspace{-10mm}
\begin{figure}[H]
	\centering
	\label{fig:c}\includegraphics[width=\textwidth,height=4.5cm,keepaspectratio]{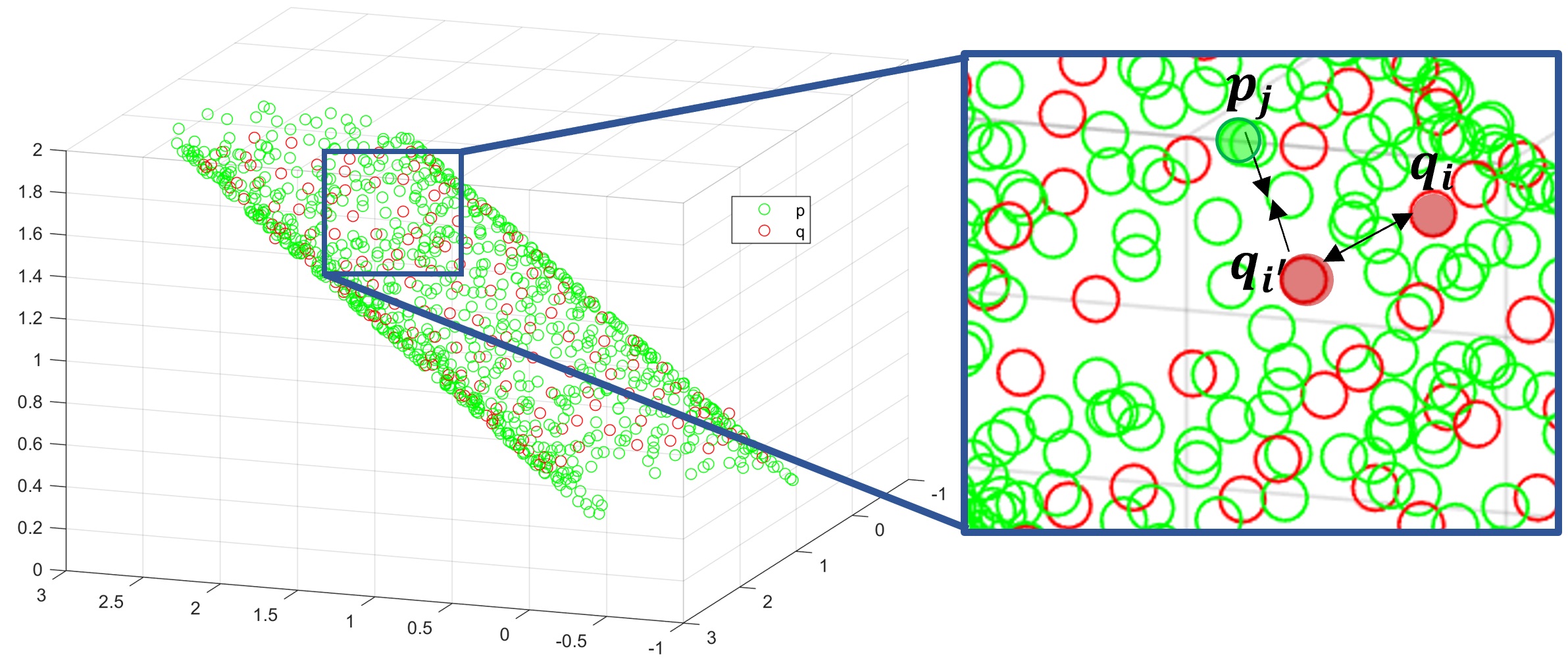} 
	\caption{Illustration of the cost function during manifold reconstruction: each point from the reconstruction set $Q$ (red points) is attracted to points in $P$ (green dots), and repelled by other points in $Q$ according to their distance.}
	\label{fig:attractionForces}
\end{figure}
In order to solve the problem with the cost function \eqref{eq:1}, we look for a point-set $Q$ that minimizes $G(Q)$. 
The solution $Q$ is found via the gradient descent iterations
\begin{equation}
q_{i'}^{(k+1)}=q_{i'}^{(k)}-\gamma_k \nabla G(q_{i'}^{(k)}),\qquad i'=1,\dots,I \,, 
\end{equation}
where the initial guess $\{q_i^{(0)}\}_{i=1}^I=Q^{(0)}$ consists of points are sampled from $P$. \\
The gradient of $G$ is given by 
\begin{equation}  \label{eq:GradG} \nabla G(q_{i'}^{(k)}) = \sum\limits_{j=1}^J {\big(q_{i'}^{(k)}-p_j\big) \alpha_j^{i'}}- \lambda_{i'}\sum\limits_{\substack {i=1 \\ i\neq i'}}^I{ \big(q_{i'}^{(k)} - q_{i}^{(k)} \big) \beta_i^{i'}}\,, 
\end{equation}
with the coefficients $\alpha_j^{i'}$ and $\beta_j^{i'}$  given by the formulas
\begin{equation}
\label{eq:alpha}
\alpha_j^{i'}  = \frac{w_{i,j}}{ \|q_i-p_j\|_{H_{\epsilon}}} \left(1-\frac{2}{h_1^2}\|q_i-p_j\|_{H_{\epsilon}}^2\right)
\end{equation}
and 
\begin{equation}
\label{eq:beta}
\beta_i^{i'} = \frac{\widehat w_{i,i'}}{ \|q_i-q_{i'}\|} \left( \left| {\frac{\partial \eta \left( \|q_i-q_{i'}\| \right) } {\partial r}}\right|  +
\frac{2\eta \left( \|q_i-q_{i'}\| \right) } {h_2^2}   \|q_i-q_{i'}\|
\right),	 
\end{equation}
for 	$i=1,...,I$, $i\ne i'$.
In order to balance the two terms in $\nabla G(q_{i'}^{(k)})$, the factors $\lambda_{i'}$ are initialized in the first iteration as
\begin{equation}  \label{eq:2}
\lambda_{i'} = -\,\frac{\bigg\|\sum\limits_{j=1}^J {\big(q_{i'}^{(k)}-p_j\big) \alpha_j^{i'}} \bigg\|}{\bigg\|\sum\limits_{i=1}^I{\big(q_{i'}^{(k)} - q_{i}^{(k)} \big) \beta_i^{i'}} \bigg\|}\,.
\end{equation}

Balancing the contribution of the two terms is important in order to maintain equal influence of the attraction and repulsion forces in $G(Q)$.
The step size in the direction of the gradient $\gamma_k$ is calculated following the procedure suggested by Barzilai and Borwein in  \cite{barzilai1988two}, as
\begin{eqnarray} \label{gamma_k} \gamma_k = \frac{\langle \bigtriangleup q_{i'}^{(k)},  \bigtriangleup G_{i'}^{(k)} \rangle} {\langle \bigtriangleup G_{i'}^{(k)},  \bigtriangleup G_{i'}^{(k)} \rangle}\,,  \end{eqnarray}
where  $\bigtriangleup q_{i'}^{(k)}  = q_{i'}^{(k)}  - q_{i'}^{(k-1)}$ and $\bigtriangleup G_{i'}^{(k)}  = \nabla G_{i'}^{(k)}  - \nabla G_{i'}^{(k-1)}$.\\

The reconstruction process is summarized in Algorithm \ref{alg:Alg1_} below:
\vspace{-6mm}
\begin{algorithm} [H]
	\caption{MLOP: Iterative Manifold Reconstruction}
	\label{alg:Alg1_}
	\begin{algorithmic}[1]
		\State {\bfseries Input:} $P=\{p_j\}_{j=1}^J \subset \mathbb{R}^n$, $\epsilon>0$
		\State{\bfseries Output:} $Q=\{q_i\}_{i=1}^I \subset \mathbb{R}^n$
		\State Initialize $Q^{(0)}$ as a subsample of $P$
		\State Estimate $h_1$ and $h_2$
		\Repeat
		\For {each $q_{i'}^{(k)} \in Q^{(k)}$}
		\State{Calculate $\nabla G(q_{i'}^{(k)})$ by assessing $\alpha_j^{i'}$, $\beta_i^{i'}$}
		\State{$q_{i'}^{(k+1)}=q_{i'}^{(k)}-\gamma_k \nabla G(q_{i'}^{(k)})$}
		\EndFor
		\Until{$\|\nabla G(q_{i'}^{(k)})\| <\epsilon$}
	\end{algorithmic}
\end{algorithm}
\vspace{-12mm}

Naturally, several changes were made to the LOP algorithm when shifting from the low-dimension to high-dimensional case. The \textbf{main enhancements of the LOP algorithm which were introduced in MLOP for high-dimensional space} can be summarized in the following list: 
\begin{enumerate}[noitemsep]
	\item The problem is reformulated in terms of looking for a new set $Q$ which will maintain the conditions in \eqref{eq:1}. This change is taken into account when taking the derivatives.
	
	\item The $L_1$ norm used in $E_1$ is replaced with the $H_{\epsilon}$, defined in \cite{levin2015between} as $\|v\|_{H_\epsilon}=\sqrt{v^2+\epsilon}$, where $\epsilon >0$ is a fixed parameter. The motivation behind this is to have a "norm" which is less sensitive to outliers. Instead of squares of errors or the absolute values of the errors, we will use an error measure that behaves as squared error for small errors and as an absolute error if the error is large. Please note that we change the norm only in the first term in \eqref{eq:1} to cope with the outliers in $P$.

	\item The norm calculation is modified to cope with high-dimensional data with noise, by using the sketching technique. For more details see Section \eqref{sec:highDimDistSec}.
	
	\item From practical reasons, we replace the fixed point iterations used in \cite{lipman2007parameterization}, with a gradient descent. The motivation behind it was to use a methodology that will allow easier theoretical analysis of the already challenging non-convex function $G$. 
	
	\item A new definition for the balancing terms $\lambda_i$ is suggested, such that the $\lambda_i$ does not change along the iterations (and there is no need to take the their derivatives). 
	
	\item Different support sizes are used when looking at the support of a given point $q_i$ with respect to $P$ and with respect to $Q$. This is natural when the number of points in $P$ and $Q$ differ. In addition, we propose a procedure for estimating these parameters (see Section \eqref{sec:optH}).
\end{enumerate}
\vspace{-3mm}
\section{Practical Details}
In Section \ref{sec:framework} we introduced the method for high-dimensional denoising and reconstruction, by optimizing a cost function that leverages the proximity to the original data and asks for quasi-uniform reconstruction. In the following two subs-sections, we will discuss several practical aspects related to robust high-dimensional distance calculation, as well as the optimal selection of the support of the weight function $w_{i,j}$.
\vspace{-6mm}
\subsection{Robust Distance Calculation in High Dimensions}
\label{sec:highDimDistSec}
The reasoning in terms of Euclidean distances, which is the cornerstone of ~Algorithm \ref{alg:Alg1_}, works well in low dimensions, e.g., for the reconstruction of surfaces in 3D, but breaks down in high dimensions once noise is present. For example, consider three points $A$, $B$ and $C$ in $\mathbb{R}^2$ (Figure \ref{fig:higDimDist} (A)), where the points $A$ and $B$ are close, whereas the point $C$ is far. Next, we embed these points in to $\mathbb{R}^{60}$ with a uniformly additive noise distribution $U(-0.2, 0.2)$ (for example in Figure \ref{fig:higDimDist} (D) we plot one of the points in $\mathbb{R}^{60}$). Unfortunately, the noise completely wipes out the signal and as a result far points cannot be distinguished from adjacent ones, see Figure \ref{fig:higDimDist} (B) (see \cite{aggarwal2001surprising, domingos2012few}). 

\forceindent To deal with this issue, we perform dimension reduction via random linear sketching \cite{woodruff2014sketching}. It should be emphasized that the dimension reduction procedure is utilized solely for the calculation of norms, and the manifold reconstruction is performed in the high-dimensional space. Given a point $x \in \mathbb{R}^n$, we project it to a lower dimension $m \ll n$ using a random matrix, $S$, with certain properties (its construction is described in detail in Algorithm \ref{alg:Alg2_}). Subsequently, the norm of  $\|S^t x\|$ will approximate $\|x\|$. Figure \ref{fig:higDimDist} (C) shows that calculating the distance in lower-dimensional space solves the distance conflicts. 

\forceindent In Algorithm \ref{alg:Alg2_} we present the details of finding the matrix $S \in \mathbb{R}^{n\times m}$. For given scattered data points $P=\{p_j\}_{j=1}^J \subset \mathbb{R}^n$ we construct matrix $S$ only once during the initialization process of ~Algorithm \ref{alg:Alg1_}. Next, given a new point $x \in \mathbb{R}^n$, its norm is approximated as $\|S^t x\|$ and utilized only for the gradient calculations in (\ref{eq:GradG}). In this paper, we choose to perform a global linear projection. However, for additional accuracy, it is possible to find a local transformation for each neighborhood.

\begin{remark}
	How should we choose the dimension \textup{m} of the space on which we project the data? First, if the dimension of the manifold $\mathcal{M}$ is known, this information can be utilized for setting \textup{m}. Alternatively, one can calculate a rough estimate, or apply a local \textup{PCA}, and use the number of the dominant eigenvalues. In our examples, the typical size of \textup{m} was set to \textup{10}.
\end{remark}

\vspace{-7mm}
\begin{algorithm} [H]
	\caption{Robust Distance Calculation in High Dimensions}
	\label{alg:Alg2_}
	\begin{algorithmic}[1]
		\State {\bfseries Input:} $P=\{p_j\}_{j=1}^J \subset \mathbb{R}^n$, $m$
		\State {\bfseries Output:} $S$ - an $n \times m$ matrix
		\State {Sample $G\in \R^{J\times m}$ with $G\sim N(0, 1)$.}
		\State {Compute $B \in \R^{n \times m}$ as $B:=P^{\rm t}G$.}
		\State {Calculate the QR decomposition of $B$ as $B = SR$, where $S \in \R^{n \times m}$ has orthonormal columns and $R\in \R^{m \times m }$ is upper triangular.}
	\end{algorithmic}
\end{algorithm}
\vspace{-12mm}

\begin{figure}[H]
	\centering
	\captionsetup[subfloat]{farskip=0pt,captionskip=0pt, aboveskip=0pt}
	\subfloat[][]{ \includegraphics[width=0.33\textwidth]{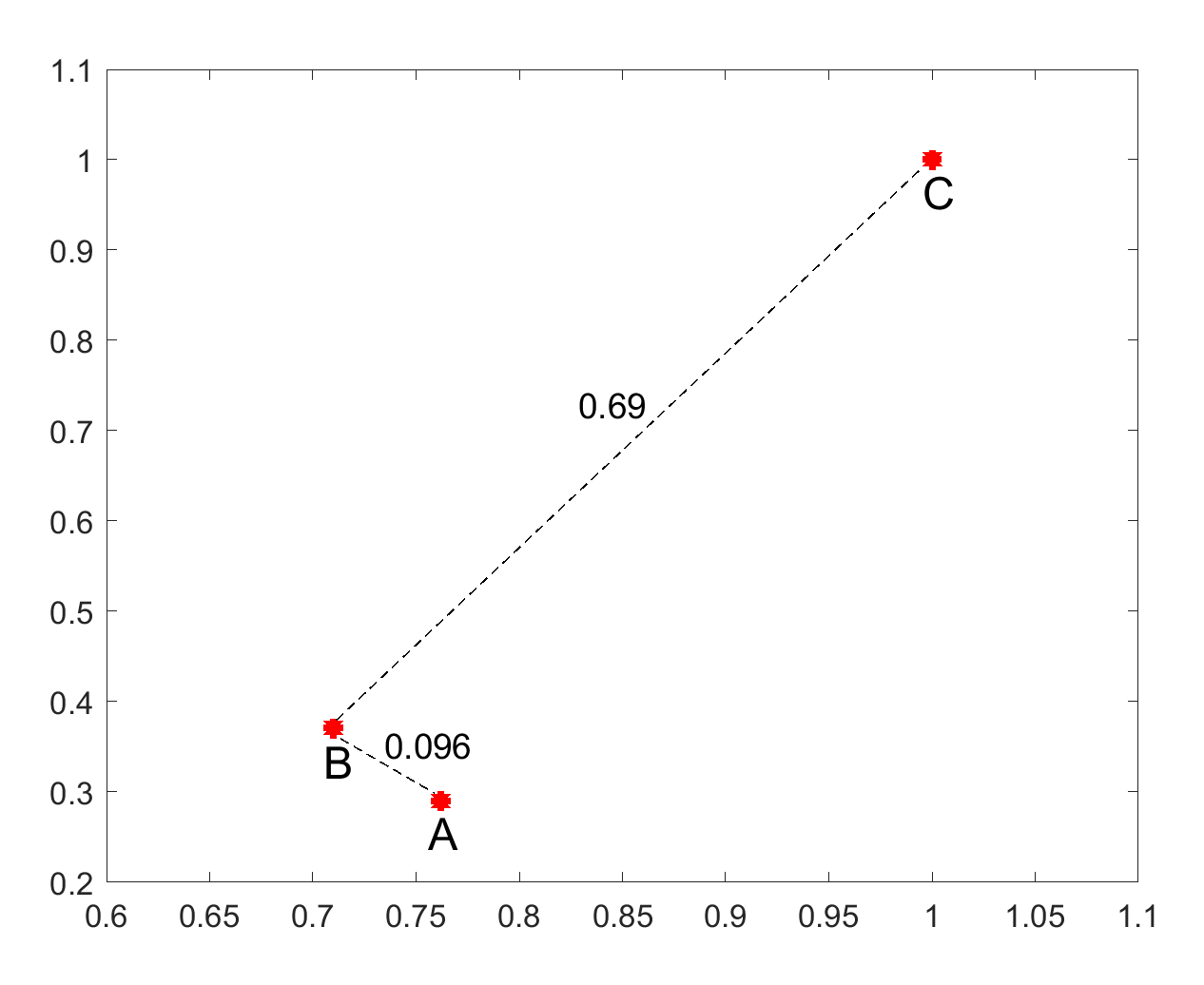} } \hspace{-1.7em}
	\subfloat[][]{ \includegraphics[width=0.33\textwidth]{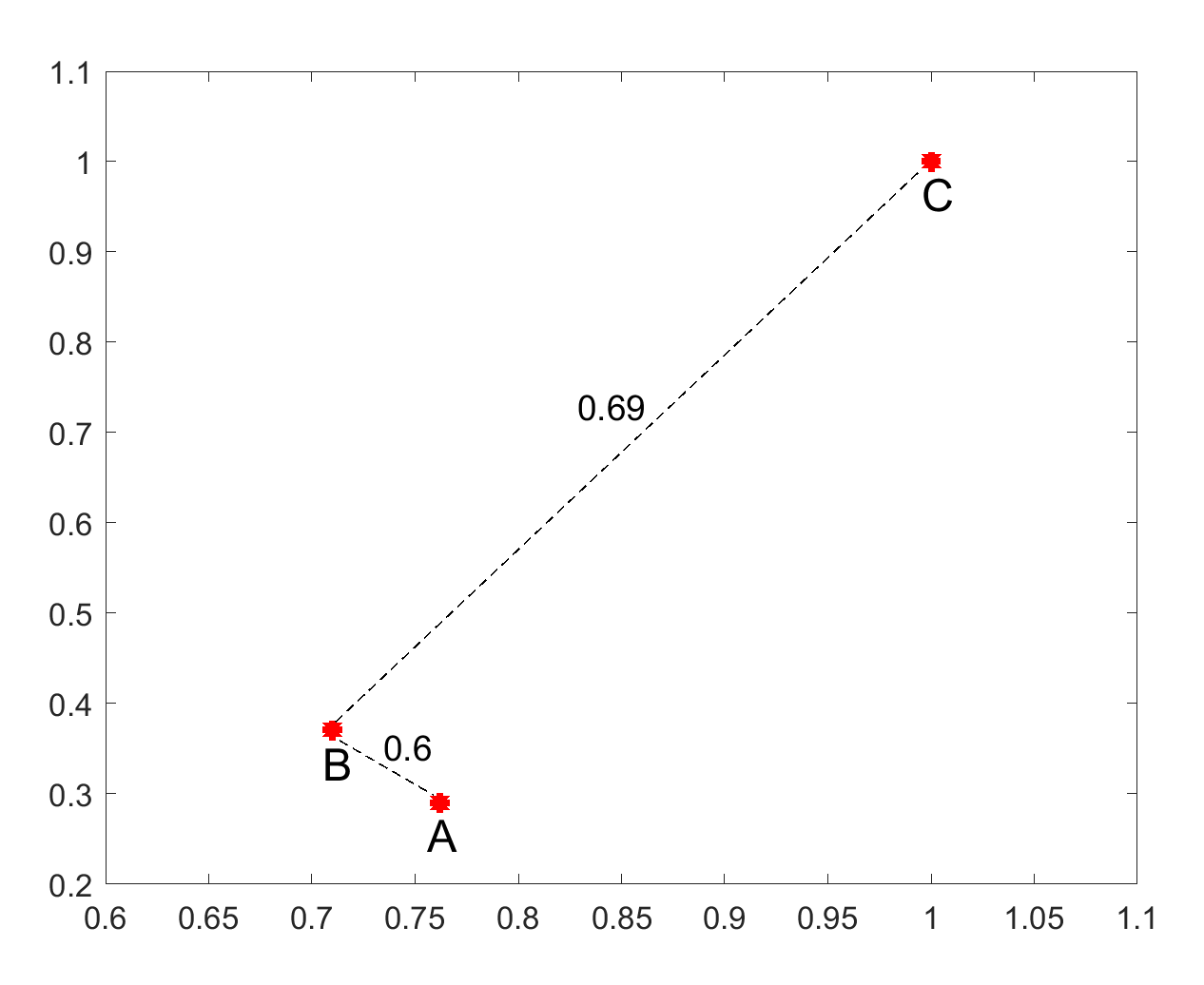} } \hspace{-1.7em}
	\subfloat[][]{ \includegraphics[width=0.33\textwidth]{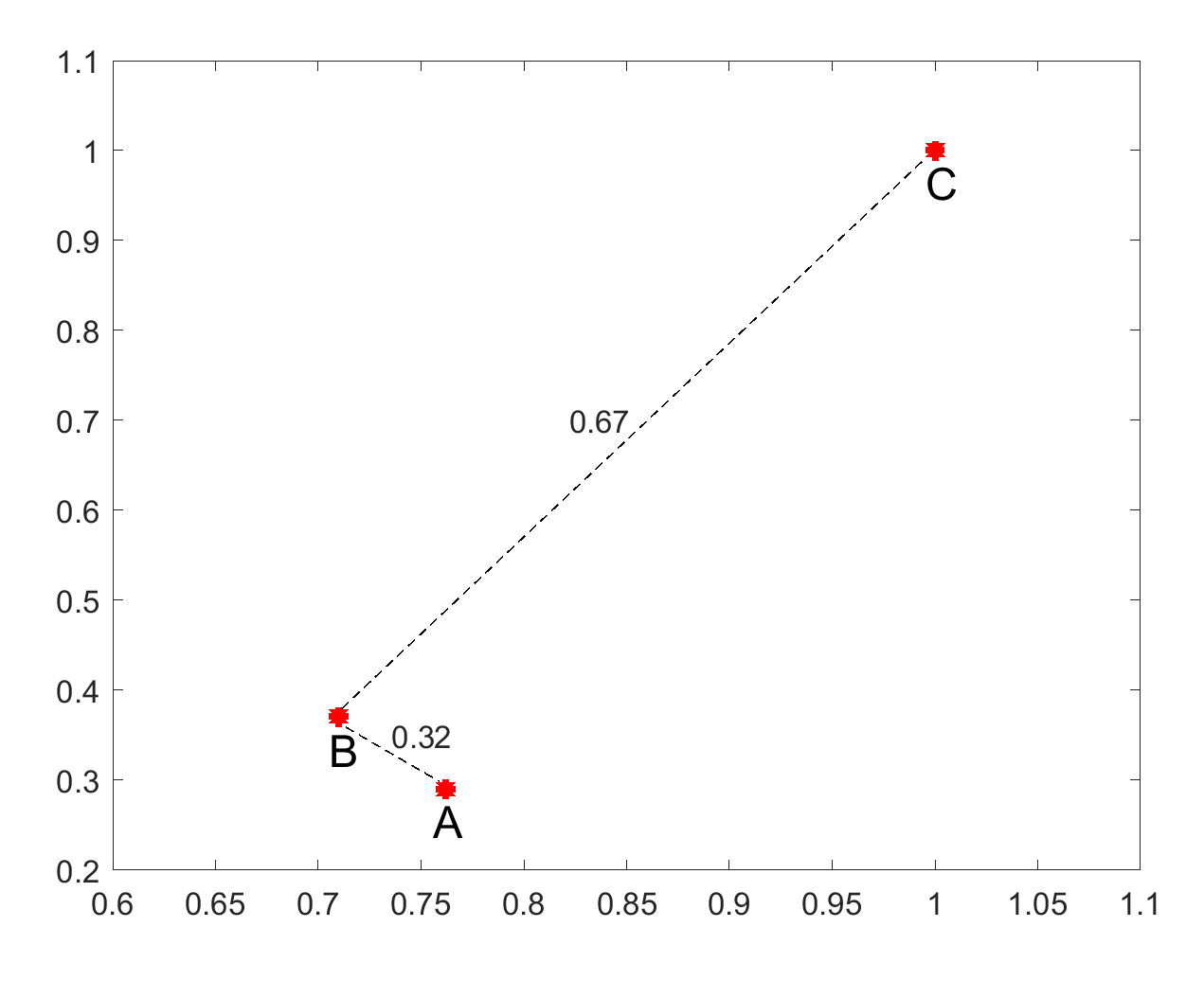} } \\[-0.7ex]
	\subfloat[][]{\includegraphics[width=0.6\textwidth]{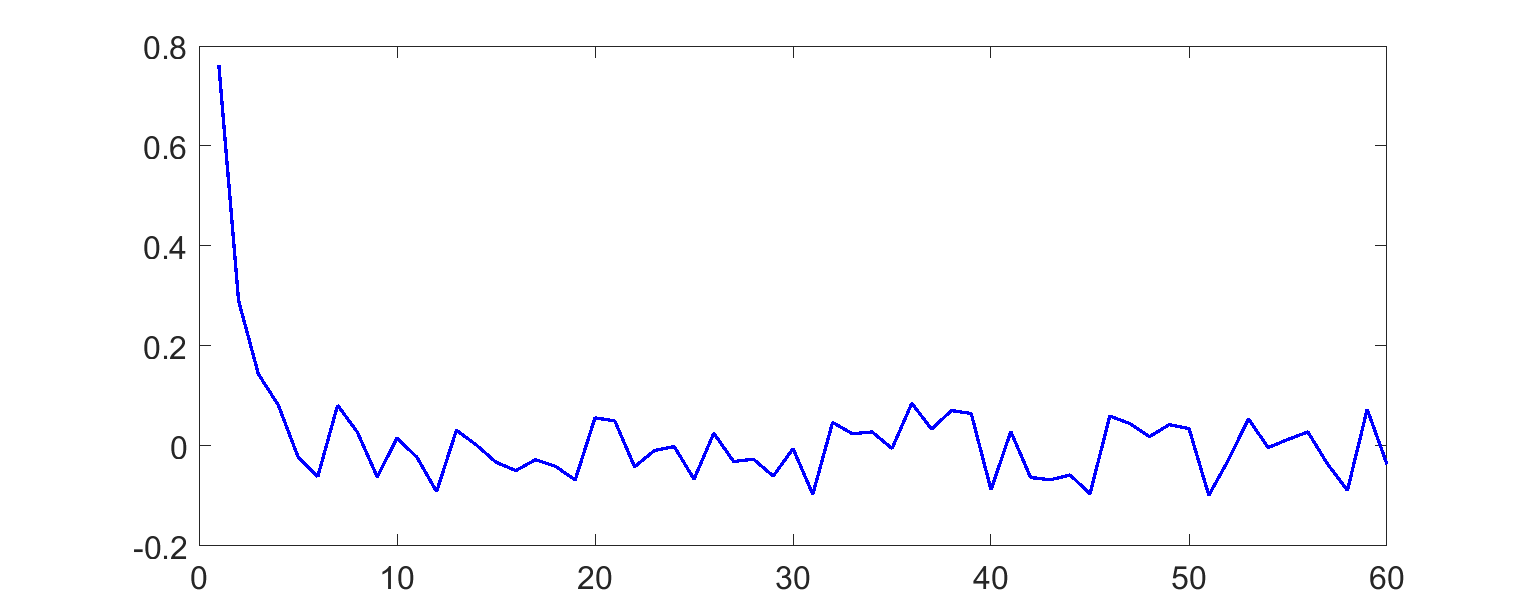}}
	\caption{Calculating distances in low-and high-dimensional space. (A) Distance calculation of points in $\R^2$. (B) Distance calculation of points in $\R^2$ embedded into $\R^{60}$ + noise $U(-0.2; 0.2)$; (C) Distance calculation of points in $\R^2$ embedded into $\R^{60}$ + noise: after sketching, (D) Point A embedded into $\R^{60}$ + noise.}
	\label{fig:higDimDist}
\end{figure}

\subsection{Optimal Neighborhood Selection}
\label{sec:optH}

In this subsection we consider the support size of the locally supported weight functions $w_{i,j}$ utilized in \eqref{eq:1} for manifold reconstruction. Specifically, given a point-set $X = \{x_k\}_{k=1}^K$, we address the problem of choosing a support size $h$ that will guarantee a sufficient amount of points from $X$ in the neighborhood of a point $q_i$ during the MLOP approximation. Although, the LOP technique has gained much popularity, and many extensions were suggested. However, the proper choice of neighboring points to be used in the reconstruction still remains an important open problem. From the one side, taking points far from the tested point can be influenced by the changing geometry of the manifold, from the other side if the neighborhood size is too small we can lose the robustness to noise property. As a result, support size selection is a critical point when dealing with a fast decaying weight function, and it is important to find an estimate to it (e.g., see the analysis for the MLS case in \cite{lipman2006error}).

\forceindent There is a high degree of freedom in choosing the points participating in the approximation since the number of data points is usually very large. Naturally, one would like to make use of these large degrees of freedom to achieve the “best” reconstruction. In what follows, we use the service centers considerations in order to approximate $h$ as a radius of the ball containing the $K$-nearest neighbors. It should be noted that naturally, we look for two parameters $h_1$, and $h_2$, defined as the support sizes of $q_i$ with respect to $P$ and $Q$, respectively. The reason for having different supports is due to the fact that the number of points in $P$ and $Q$ can differ, and this should be reflected in the choice of their support size. As will be demonstrated in the numerical examples section, our approach outperforms the heuristic choice of support size in approximation quality and stability.

\forceindent The support sizes $h_1$, and $h_2$ are closely related to the fill-distance of the $P$ points and the $Q$ points. Let $J$ and $I$ be the sizes of the sets $P$, and $Q$ respectively. In case $I\leq J$, each $q_i$ can be viewed as a service center that serves approximately $\nu= \floor*{\frac{J} {I}}$ points from the $p_j$'s.  We use this observation to calculate the fill-distance of $P$, then estimate the support that guarantees at least $\nu$ points in the neighborhood of $p_j$, as well as the practical support size of the Gaussian $w_{i,j}$ (see the illustration in Figure \ref{fig:RadIlust}).

Unlike the standard definition of fill-distance in scattered data function approximation \cite{levin1998approximation}, we introduce
\begin{definition}\label{def:def1}
	The fill-distance of the set $P$ is 
	\begin{eqnarray} \label{h_0} h_0=\maxmath_{y\in M}\minmath_{p_j\in P}\|y-p_j\|.		\,. \end{eqnarray}
\end{definition}

\begin{definition}\label{def:def2}
	Given two point-clouds $P=\{p_j\}_{j=1}^J \subset \mathbb{R}^n$ and $Q=\{q_i\}_{i=1}^I \subset \mathbb{R}^n$, situated near a manifold $\mathcal{M}$ in $\mathbb{R}^n$, such that their sizes obey the constraint $I\leq J$, denote $\nu= \floor*{\frac{J} {I}}$. Then we say that the radius that guarantees approximately $\nu$ points from $P$ in the support of each point $q_i$ is $\hat{h}_0=c_1 h_0$, with $c_1$ given by 
	\begin{eqnarray} \label{c_1} c_1=\text{argmin} \{c: \#(\bar{B}_{ch_0}(q_i) \cap P)\geq \nu,\,  \forall q_i\in Q\}\,.\end{eqnarray}
	where $\#(B_r(x) \cap P)$ is the number of points in a ball $B_r(x)$ of radius $r$ centered at the point $x$.
\end{definition}

\begin{remark}\label{def:def3}
	Let $\sigma$ be the variance of a Gaussian $w(r) = e^{-\frac{r^2}{\sigma^2}}$. For the normal distribution, four standard deviations away from the mean account for $99.99\%$ of the set.
	In our case, by the definition of $w_{i,k}$, since $h$ is the square root of the variance, $4\sigma= 4\frac {h}{\sqrt{2}}= 2 \sqrt{2}h_1$ covers $99.99\%$ of the support size of $w_{i,k}$.
\end{remark}

The following theorem indicates how the parameters $h_1$ and $h_2$  should be selected.

\begin{Theorem}\label{lma0}
	Let $\mathcal{M}$ be a $d$-dimensional manifold in $\mathbb{R}^n$. Suppose given two point-clouds $P=\{p_j\}_{j=1}^J \subset \mathbb{R}^n$ and $Q=\{q_i\}_{i=1}^I \subset \mathbb{R}^n$ situated near a manifold $\mathcal{M}$ in $\mathbb{R}^n$, such that their sizes obey the constraint $I\leq J$, and let $\nu= \floor*{\frac{J} {I}}$. Let $w_{i,j}$ be the locally supported weight function given by $w_{i,j} = \exp\big\{\!-\|q_i-p_j\|^2/{h^2}\big\}$. Then a neighborhood size of $h = 2 \sqrt{2}\hat{h}_0$ guarantees $ 2^{1.5d}\nu$ points in the support of $w_{i,j}$, where $\widehat{h}_0=c_1h_0$, with $c_1$ given by  \textup{\eqref{c_1}}.
\end{Theorem}

\begin{proof}
	
	Given a point $q_i$ we look for the amount of points from $P$ in the support of $w_{i,j}$. Using Remark \ref{def:def3} we can estimate the support size of $w_{i,j}$ as $4\sigma$, where $4\sigma= 2 \sqrt{2}h_1$. We denote the amount of points from $P$ in the support of $q_i$ by $S_{4\sigma}$. In what follows we assume that the proportion of the number of points in a support does not change with radius changes. Thus, $S_{4\sigma}$ can be determined from the ratio of the volume to the amount of served points: $ \frac {V_1}{V_2}= \frac {S_{\sigma}}{S_{4\sigma}}$, where the volume of a ball with radius $\hat{h}_0$ in $\mathbb{R}^d$ is $V_1 = \pi^{d/2} \hat{h}_0^d/c(d)$, and the volume of a ball with radius $4\sigma$ is $V_2=\pi^{d/2} {(4\sigma)}^d/c(d)=2^{1.5d} \pi^{d/2} \hat{h}_0^d/c(d)$ (where $c$ is Euler's gamma function). Thus, $S_{4\sigma}=\nu \frac {V_2}{V_1} = 2^{1.5d}\nu$.
\end{proof}
\vspace{-12mm}
\begin{figure}[H] 
	\centering
	\label{fig:g}\includegraphics[width=\textwidth,height=4.5cm,keepaspectratio]{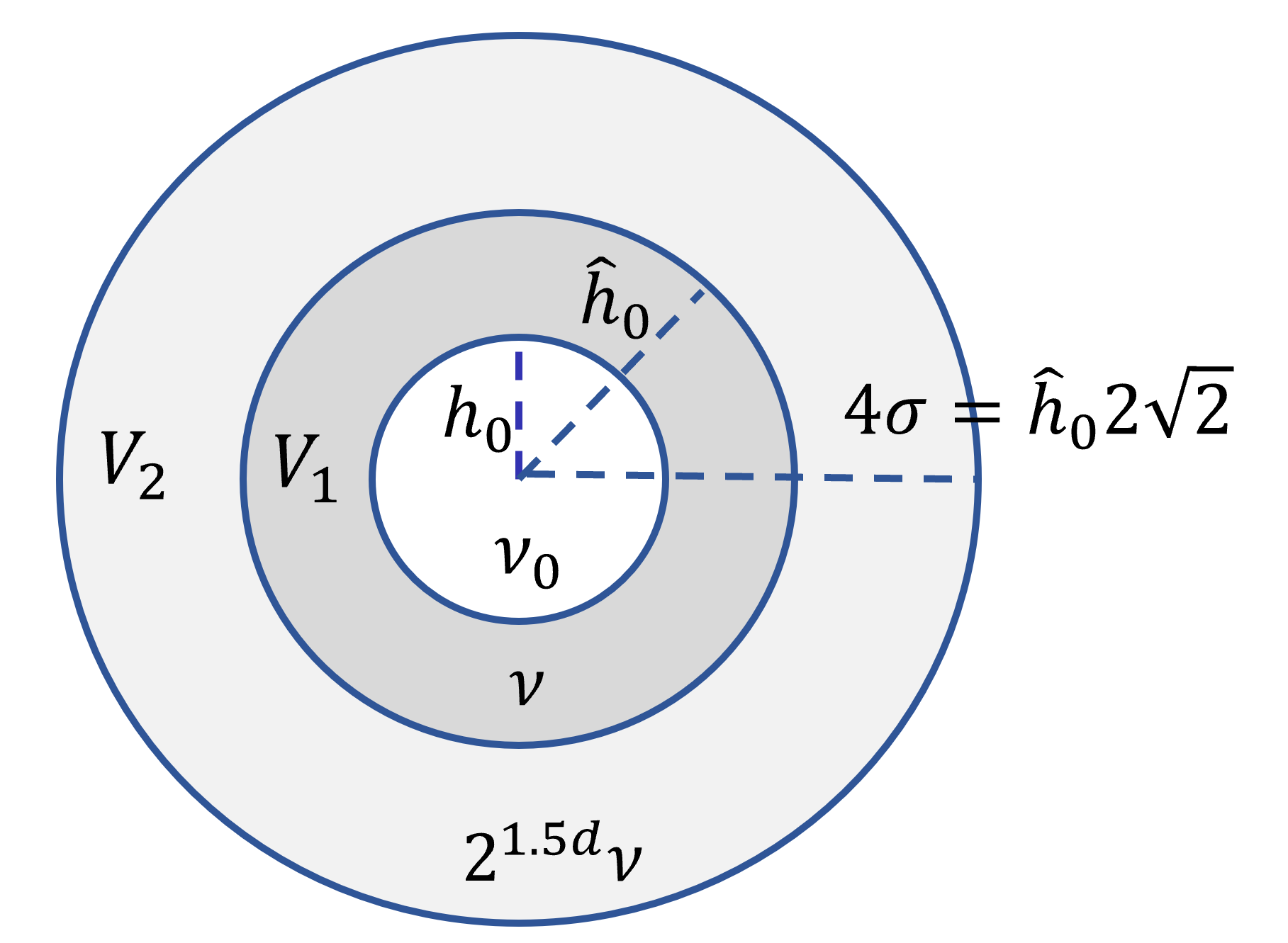}
	\caption{Scheme of the fill-distance and the size of the support of the weight function. $h_0$ is the radius that guarantees at least one point $p_j$ in the support of $q_i$, $\hat{h}_0$ guarantees 
		$\nu$ points, while the real number of points in the support is $2^{1.5d}\nu$.}
	\label{fig:RadIlust}
\end{figure}
\vspace{-5mm}
\begin{corollary}
	Let $P$ and $Q$ be as defined in Theorem \textup{\ref{lma0}}, and assume $J< I$, then the number of $Q$ points in the support of each $p_j \in P$ is $2^{1.5d}\nu$.
\end{corollary}
\begin{proof}
	Each $p_j$ can be viewed as a service center that serves approximately $\nu=I/J$ points $q_i$ from $Q$.  All the preceding definitions remain valid, except that the roles of $P$ and $Q$ are switched. Namely, $h_0$ is the fill-distance of the set $P$ within the set $Q$, $\hat{h}_0$ guarantees $\nu$ points from $Q$ near each point from $P$, and the actual number of $Q$ points in the support of $P$ is $2^{1.5d}\nu$.
\end{proof}

\begin{remark}
	\label{practical_hs}
	\textbf{Practical considerations for the support size calculations.} As mentioned above, given a point $q_i$ we estimate two different support sizes $h_1$ and $h_2$ with respect to the sets $P$ and $Q$ to be used in \ref{eq:1}. Assume $I < J$, then $h_1$ is set to be $\hat{h}_0$, which is calculated using definition \ref{def:def2}. Since we don't have any knowledge about the uniformity of distribution of the $Q$ points over $\mathcal{M}$, we estimate $h_2$ as follows. We sample $I$ points uniformly from $P$, and denote this set by $Q^{rand}$. Next, we estimate $h_2$ as $\hat{h}_0$ using definition \ref{def:def2}, when substituting both of the sets $P$ and $Q$ to be $Q^{rand}$. This gives a rough estimation of $h_2$ in the scenario when the $Q$ points are equality distributed over $\mathcal{M}$.
	
\end{remark}

\begin{remark} \label{rem:reach_def}
	The \textup{reach} $\tau_M$ of $\mathcal{M} \subset \R^n$ is defined as the largest number such that any point at distance less than $\tau_M$ from $\mathcal{M}$ has a unique nearest point on $\mathcal{M}$ \cite{federer1959curvature}. We note that \textup{$h$} should be smaller than the reach $\tau_M$ of the manifold $\mathcal{M}$. The reason for this is to prevent a situation where the weighted summations used in the cost function \textup{(\ref{eq:1})} may be influenced by points in another branch of $\mathcal{M}$ if this constraint is violated. 
\end{remark}

\section{Main Results}
\label{sec:TheoAnalysis}

Although LOP became popular for surface reconstruction, very important theoretical aspects of the methodology didn't gain attention. The main goal of the analysis presented in this section is to complete the missing parts of the puzzle for the high-dimensional case. We will prove the convergence of the MLOP method, order of approximation, convergence rate as well as its complexity (presented in Theorem \ref{thm:conv}, Theorem \ref{thm:approx_order} and Theorem \ref{thm:convRate}, respectively). In addition, we will discuss the uniqueness of the MLOP solution (see Subsection \ref{sec:Uniqueness}).

\subsection{Convergence to a Stationary Point}

We are now ready to state our main convergence theorem. The fact that the cost function is non-convex poses a challenge for the proof of the convergence of the proposed method. First, we define $h$ as described in Section \ref{sec:optH} and assume that the $h$-$\rho$ condition, defined above, is satisfied. Next, we utilize the following general non-convex convergence theorem presented in \cite{lee2016gradient} to prove the convergence of our method.

\begin{Theorem}\label{thm:convLeeEtAl}
	Let $f:\mathbb{R}^d \to \mathbb{R}$, not necessarily convex, be twice continuously differentiable and has Lipschitz gradient, with constant L, i.e., $\|\nabla f(x)-\nabla f(y) \| \leq L \|x-y\|$. Let its the gradient descent of $f$ be $x^k=x^{k-1}- \alpha \nabla f(x^{k-1})$, with bounded step size $0<\alpha<1/L$. Suppose, all saddle points of the function $f$ are strict-saddle (i.e., for all critical points $x^*$ of $f$, $\lambda_{\minmath} {\nabla^2 (f(x^* ))}<0$). Then the gradient descent with random initialization and sufficiently small constant step size converges almost surely to a local minimizer or to minus infinity. i.e., if $x^*$ is a strict saddle then $\Prmath(\limmath \text{ } x_k=x^*) = 0$.
	
\end{Theorem}
We also recall the following theorem on eigenvalue bounds, due to Iyengar et al. \cite{iyengar2000estimating}.
\begin{Theorem}\label{thm:eigValueBounds}
	The highest and lowest eigenvalues of a self-adjoint matrix $X$, with entries $x_{i,j}$, lie in the range
\vspace{-12mm}

	\begin{eqnarray*} \lambda_{\minmath}, \lambda_{\maxmath} \in [l, u]\,,\end{eqnarray*} 
\vspace{-6mm}
	where \begin{eqnarray*} l = \underset{i \in I}{\minmath} \left( x_{i,i} - \sum \limits_{j} |x_{i,j}| \right) and   \quad   
		u = \underset{i \in I}{\maxmath} \left (x_{i,i} + \sum \limits_{j} |x_{i, j}| \right) \,.\end{eqnarray*} 
\end{Theorem}
\begin{Theorem}[Convergence to a stationary point]\label{thm:conv}
	Let $\mathcal{M}$ be a $d$-dimensional manifold in $\mathbb{R}^n$, where $d$ is an unknown intrinsic dimension. Suppose that the scattered data points $P = \{{p_j }\}_{j =1}^J$ were sampled near the manifold $\mathcal{M}$, $h_1$ and $h_2$ are set as defined in Section \textup{\ref{sec:optH}}, and the $h$-$\rho$ set condition is satisfied with respect to $\mathcal{M}$. Let the points $Q^{(0)}=\{q_i^{(0)} \}_{i=1}^I$ be sampled from $P$. Then the gradient descent iterations \textup{\eqref{eq:1}} converge almost surely to a local minimizer $Q^*$.
\end{Theorem}
\begin{proof} 
	We proceed by verifying that the conditions of Theorem~\ref{thm:convLeeEtAl} hold. At a high level, our proof consists of the following steps: 
	\begin{enumerate}[noitemsep]
		\item Calculate the Hessian of the cost function \eqref{eq:1}.
		\item Bound the eigenvalues of the Hessian.
		\item Show that the minimal eigenvalue is negative.
		\item Bound the norm of the Hessian.
	\end{enumerate}
	We rephrase the minimization problem from \eqref{eq:1} by writing $E_1$ and $E_2$ in a matrix form as
	\begin{eqnarray*}E_1=\begin{pmatrix}1\\ \vdots \\1\end{pmatrix}^t 
		\begin{pmatrix}\|q_1-p_1\| w_{1,1}& \hdots & \|q_1-p_J\| w_{1,J}
			\\ \vdots &\vdots & \vdots 
			\\ \|q_I-p_1\| w_{I,1}& \hdots & \|q_I-p_J\| w_{I,J}\end{pmatrix} \begin{pmatrix}1\\ \vdots \\1\end{pmatrix}\,, \end{eqnarray*}
	
	\begin{eqnarray*}E_2=\begin{pmatrix} \lambda_1\\ \vdots \\\lambda_I\end{pmatrix}^t 
		\begin{pmatrix}0 & \eta (\|q_1-q_2\|) w_{1,2} & \hdots & \eta(\|q_1-q_I\|) w_{1,I}
			\\ \eta (\|q_2-q_1\|) w_{2,1} & 0 &  \hdots & \eta(\|q_2-q_I\|) w_{2,I}
			\\ \vdots &\vdots & \vdots  & \vdots
			\\ \eta(\|q_I-q_1\|) w_{I,1}& \hdots & \eta(\|q_I-q_{I-1}\|) w_{I,I-1} & 0\end{pmatrix} \begin{pmatrix}1\\ \vdots \\1\end{pmatrix}\,. \end{eqnarray*} 
	The cost function is rewritten as
	\begin{eqnarray*}G(Q)=\vec{1}^t \Phi \vec{1}+ \vec{\Lambda}^t \Psi\vec{1} \,, \end{eqnarray*}
	where $\phi_{i,j} = \|q_i-p_j\| w_{i,j}$ are the entries of $\Phi$, $\psi_{i,j} = \eta (\|q_i-q_{i'}\|) \hat w_{i,i'}$ are the entries of $\Psi$, and the vector of balancing parameters $\vec{\Lambda} = (\lambda_1, \hdots, \lambda_I)$ is defined in \eqref{eq:2}.
	
	\forceindent The proof relies on the fact that the weights $w_{i,j}$ are defined by rapidly decreasing functions with respect to a point $q_i \in Q$. Although the weight function $w_{i,j}$ in definition \eqref{eq:1} does not have compact support, for practical reasons it can be assumed that the Gaussian with $4\sigma$ covers $99\%$ of the support size. As a result, the matrices $\Phi$ and $\Psi$ are sparse, and the number of their non-zero entries depend on the support size of $w_{i,j}$. Following Definition \ref{def:def3}, we estimate the number of non-zero entries in each row of the matrices $\Phi$ and $\Psi$, in the $k$th iteration of our algorithm, as
	\begin{align*}&	\Phi_{q_i^{(k)}} = \#\{B_{h}(q_i^{(k)})  \cap P\}\,, \\
	&\Phi_{p_j^{(k)}} = \#\{B_{h}(p_j^{(k)})  \cap Q^{(k)}\}\,, \\
	&\Psi_{q_i^{(k)}} = \#\{B_{h}(q_i^{(k)})  \cap Q^{(k)}\}\,,
	\end{align*}
	where $B_{h}(x)$ is a ball centered at $x$ with radius $h$.
	
	Using these definitions, we calculate the Hessian and its eigenvalues for our cost function in \eqref{eq:1},
	\begin{eqnarray*}H=\nabla^2G(Q) = \nabla^2E_1 + \Lambda \nabla^2E_2\,.\end{eqnarray*} 
	For simplicity, we denote $r_{i,j} = q_i-p_j$; then with $w_{i,j} = \exp\big\{\!-\|q_i-p_j\|^2/{h_1^2}\big\}$,  $\frac{\partial E_1}{\partial q_i}$ can be rewritten as 
	\begin{eqnarray*}	
		\frac{\partial E_1}{\partial q_i} = \sum\limits_{j=1}^J {\frac{r}{\|r_{i,j}\|} \left(1-\frac{2}{h_1^2}\|r_{i,j}\|^2\right) w_{i,j}} \,.
	\end{eqnarray*}
	We notice that, by definition, $\frac{\partial^2 E_1}{{\partial q_i} {\partial q_i'}} = 0\,,$
	and by the chain rule we have 
	\begin{eqnarray*}	\frac{\partial^2 E_1}{\partial q_i^2} =\sum\limits_{j=1}^J {a(r_{i,j}) w_{i,j}}\,,\end{eqnarray*} 
	where $a(r) = -\frac{2}{h_1^2}\|r\| \left(1+\frac{2}{h_1^2}\|r\|^2 \right) <0$. \\
	For the second term in expression (\ref{eq:1}), we denote  $\hat{r}_{i,i'} = q_i-q_i'$, and recall that $\eta(r) = \frac {1} {r^3}$.
	Then the first derivative of $E_2$ is 
	\begin{eqnarray*}	\frac{\partial E_2}{\partial q_i} = \sum\limits_{i'=1}^I { \left(-\frac{\hat{r}_{i,i'}}{\|\hat{r}_{i,i'}\|^5} - \frac{2\hat{r}_{i,i'}}{3h_2^2\|\hat{r}_{i,i'}\|^3}\right) \hat w_{i,i'}}\,. \end{eqnarray*}
	The second derivatives can be expressed as 
	\begin{eqnarray*}	\frac{\partial^2 E_2}{{\partial q_i} {\partial q_i'}} =-b(\hat{r}_{i,i'}) \hat w_{i,i'} \,,
	\end{eqnarray*}
	where $b(\hat{r})= \frac{4}{\|\hat{r}\|^5}+ \frac{3\frac{1}{3}}{h_2^2\|\hat{r}\|^3} + \frac{4}{3h_2^4\|\hat{r}\|} >0$, 
	and 
	\begin{eqnarray*}	\frac{\partial^2 E_2}{\partial q_i^2} =\sum\limits_{i'=1}^J b(\hat{r}_{i,i'})\hat w_{i,i'} \,.
	\end{eqnarray*}
	Thus, 
	\fontsize{9}{8}
	\begin{eqnarray*}	H=
		\begin{pmatrix} \sum\limits_{j=1}^J {a(r_{1,j}) w_{1,j}}+\lambda_1 \sum\limits_{i'=1}^I {b(\hat{r}_{1,i}) w_{1,i'} }; 
			
			& -\lambda_1 b(\hat{r}_{1,2]}) w_{1,2};  & \hdots & -\lambda_1 b(\hat{r}_{1,I})w_{1,I}
			
			\\ \vdots &\vdots & \vdots  & \vdots \\
			
			-\lambda_I b(\hat{r}_{I,1}) w_{I,1};& \hdots & -\lambda_I b(\hat{r}_{I,I-1} ) w_{I,I-1}; & 
			\sum\limits_{j=1}^J {a(r_{I,j}) w_{I,j} }+\lambda_I \sum\limits_{i'=1}^I {b(\hat{r}_{I,i'})\hat w_{i,i'}} \end{pmatrix} 
		\,. \end{eqnarray*} 
	\normalsize
	
	Let us check that the eigenvalues $\lambda_{\minmath}$, and $\lambda_{\maxmath}$ of the MLOP Hessian $H \in \mathbb{R}^{I \times I}$ are bounded and negative.
	By Theorem~\ref{thm:eigValueBounds}, the eigenvalues of $H$ belong to the range $\lambda_{\minmath}, \lambda_{\maxmath} \in [l, u]$, 
	where in our case
	\begin{eqnarray*} l=\underset{i \in I}{\minmath} \left( \sum\limits_{j=1}^J {a(r_{i,j}) w_{i,j}}+\lambda_i \sum\limits_{i'=1}^I {b(\hat{r}_{i,i}) \hat w_{i,i'} }   - \sum\limits_{i'=1}^I |{\lambda_i b(\hat{r}_{i,i']}) \hat w_{i,i'}}| \right )\,.\end{eqnarray*} 
	Let $h = min(h_1, h_2)$. Using the expressions for $a(r)$ and $b(r)$, and the fact that from Definition \ref{def:def3} $\|r\|=4\sigma=\frac{4h}{\sqrt{2}}$, it can be verified that $0 < \minmath (b(r))  \leq \frac{c_1}{h^5}$, $\minmath (a(r))  \leq \frac{-c_2}{h\sqrt 2}$, $\maxmath (a(r)) \leq 0$, where $c_1, c_2$ are constants and $c_1, c_2 >0$.
	Thus, since $\lambda_i <0$ from \eqref{eq:2}, and the number of points from $P$ and $Q$ in the support of $q_i$ estimated by $\Phi_{q_i^{(k)}}$ and $\Psi_{q_i^{(k)}}$, respectively, we have
	\begin{equation} u  \leq - \frac{c_2}{h \sqrt 2} \underset{i \in I}{\maxmath} (\Phi_{q_i^{(k)}}) <0 \,, \end{equation}
	\begin{equation}  \label{eq:l} l  \leq  - \frac{c_2}{h\sqrt 2} \underset{i \in I}{\maxmath} (\Phi_{q_i^{(k)}}) - \frac{2c_1}{h^5} \underset{i \in I}{\maxmath} (|\lambda_i|)  \underset{i \in I}{\maxmath}(\Psi_{q_i^{(k)}}-1) <0 \,. \end{equation}
	Since the eigenvalues are negative, all saddle points of the MLOP target function are strict-saddle, and the second condition of Theorem~\ref{thm:convLeeEtAl} holds. Let us also check that the first condition in Theorem~\ref{thm:convLeeEtAl} is satisfied, i.e., that the norm of the Hessian is bounded: $\| H \| \leq L$, and find $L$.
	Indeed, 
	\begin{eqnarray*}\|H\|_2=\lambda_{\maxmath} ({H'H})=\lambda_{\maxmath} (H^2)=\maxmath \{ \lambda ^2 \mid \lambda \text{ is an eigenvalue of } H\} =\maxmath \{ \lambda_{\maxmath}^2,\lambda_{\minmath}^2\} \,, \end{eqnarray*} 
	so the required bound holds with $L= \maxmath \{ \lambda_{\maxmath}^2,\lambda_{\minmath}^2\}  \leq \maxmath \{ u^2, l^2\} = l^2 $. 
	
	\forceindent To summarize, all the conditions of Theorem~\ref{thm:convLeeEtAl} are satisfied. It follows that the gradient descent with random initialization and a sufficiently small constant step size converges almost surely to a local minimizer or minus infinity. 
\end{proof}

\subsection{Order of Approximation}

The support size of the locally supported function defining the weight function $w_{i,j}$ which is tightly related to the fill-distance of available sample data $P$, plays an important role in the order of approximation of the MLOP algorithm. The following theorem guarantees an $O(h^2)$ order of approximation, which is asymptotic as $h \rightarrow 0$. Here, $h = \max (h_1, h_2)$, where $h_1$ and $h_2$ are defined in Remark \ref{practical_hs}.


\begin{Theorem}[Order of approximation]\label{thm:approx_order}
	Let $P=\{p_j\}_{j=1}^J$ be a set of points that are sampled \textup{(}without noise\textup{)} from a $d$–dimensional $C^{2}$ manifold $\mathcal{M}$, and satisfy the $h$-$\rho$ condition. Then for a fixed $\rho$, and a finite support of size $h$ of the weight functions $w_{i,j}$, the set of points $Q$ defined by the \textup{MLOP} algorithm has an order of approximation $O(h^{2})$ to $\mathcal{M}$.
\end{Theorem}
\begin{proof}
	
	We break the proof into the following steps.
	\begin{enumerate}
		
		\item \textbf{The MLOP cost function can be rewritten in matrix form as $AQ=R$}. We look for a solution $Q$ that will minimize the cost function in \eqref{eq:1}, i.e., such that the gradient $\nabla G(Q)=0$. Thus equation \eqref{eq:GradG} can be recast as a system of equations
		\begin{equation} \label{eq:5} (1-\tau_{i'} ) q_{i} + \tau_{i}\sum\limits_{i' \in I\setminus\{i\}} q_{i'}^{(k)}  \frac{\beta_i^{i'}}{\sum\limits_{i' \in I\setminus\{i\}} \beta_i^{i'}} = \sum\limits_{{j=1}}^J p_j \frac{\alpha_j^{i'}}{\sum\limits_{j\in J} \alpha_j^{i'}} \,,\end{equation}
		where we express $\lambda_{i'}$ in the form $\displaystyle \lambda_{i'} = \tau_{i'} \frac {\sum\limits_{j\in J} \alpha_j^{i'}}  {\sum\limits_{i' \in I\setminus\{i\}} \beta_i^{i'}}$.\\
		As a result, the problem can be written in matrix form as $AQ=R$, where both $A$, and $R$ depend on $Q$. 
		In the new notations, we need to show that the points $Q=A^{-1}R$ lie at a distance of $O(h^{2})$ from $\mathcal{M}$.

		\item \textbf{The $R$ term has order of approximation of $O(h^{2})$ to $\mathcal{M}$}. Let $J_k$ be the indices of points from $P$ which lie at the distance $h$ from a given poin $q_{i'}$ (the set is not empty due to the optimal neighborhood selection in Subsection \ref{sec:optH}). Let $t$ be the index of the closest point in $\{p_j\}_{j \in J_k}$ to the projection of $q_{i'}$ on the manifold $\mathcal{M}$ (Figure \ref{fig:Approx_order} left), and $T$ be the tangent space to $\mathcal{M}$ at that point. Then the sum $\sum\limits_{{j=1}}^J p_j \frac{\alpha_j^{i'}}{\sum\limits_{j\in J} \alpha_j^{i'}} $, is a local convex combination of points $p_k$ within a distance $h$
		from $q_{i'}$, and thus it also lies in $T$, which is affine.
		 Since $\mathcal{M}$ is $C^{2}$, $T$ approximates $\mathcal{M}$ in the order of $O(h^{2})$, the right hand side of \eqref{eq:5} can be written as $F+O(h^{2})$, where $F = \{f_{i}\}_{{i}\in I}$ are points on $\mathcal{M}$. 
		Thus, $AQ = F + O(h^{2})$.

		\item \textbf{Then norm of the matrix $A^{-1}$, $\|A^{-1}\|_\infty$ and its entries $(A^{-1})_{l,m}$ are bounded}.
		For $\tau_{i} \in [0, 0.5)$, the matrix $A$ is strictly diagonally dominant and therefore we can bound $\|A^{-1}\|_\infty  \leq c_1(\tau_{i})$, as well as $|(A^{-1})_{l,m}| < c_2(\tau_{i})$ for two points $q_l$ and $q_m$ lying at a distance of at least $h$, where the influence of distant points decays exponentially with distance. We also note that since the rows of A sum up to one, so do the rows of
		$A^{-1}$.

		\item \textbf{The MLOP reconstruction is of order $O(h^{2})$ to the manifold}.
		The MLOP reconstruction can be written as $Q=A^{-1}F+O(h^{2})$, where each element of $(A^{-1}F)_{i'}$ is the affine average of $f_i$ over the manifold, with exponentially decaying weights $w_{i,j}$.
		Let $T$ be the tangent space to the manifold $\mathcal{M}$ at the point $f_{i'}$, and let $t_i$ be the projection of $f_{i}$ on $T$ (Figure \ref{fig:Approx_order} right). If we rewrite $f_{i}$ using its projection as $f_i = t_i + r_i$, it follows that
		$\displaystyle (A^{-1}F)_{i'} = \sum\limits_{i\in I} {A^{-1}_{i',i}(t_i+r_i)} = \sum\limits_{i\in I} {A^{-1}_{i',i}t_i} + \sum\limits_{i\in I} {A^{-1}_{i',i}r_i}$. We would first like to show that $\|\sum\limits_{i\in I} {A^{-1}_{i',i}t_i} - f_{i'}\| = O(h)$, and since $\sum\limits_{i\in I} {A^{-1}_{i',i}t_i}$ is on $T$, and $T$ approximates the manifold with $O(h^2)$, it will follow that $\sum\limits_{i\in I} {A^{-1}_{i',i}t_i}$ is of order $O(h^2)$ distance from $\mathcal{M}$. In addition, we show that $\sum\limits_{i\in I} {A^{-1}_{i',i}(r_i)} = O(h^2)$.
		
		In more details:
		
		\begin{enumerate}
			\item For a given $q_{i'}$, we denote by $I_k$ its $q_i$ neighbors at the distance $\|q_i - q_{i'}\| \in [kh, (k+1)h]$. We use the fact that the sum of the rows of $A^{-1}$ equals one, and rewrite and estimate $\sum\limits_{i\in I} {A^{-1}_{i',i}t_i}$ as 
			\begin{equation}\|\sum\limits_{i\in I} {A^{-1}_{i',i}t_i} - f_{i'}\|=\|\sum\limits_{i\in I} {A^{-1}_{i',i} (t_i - f_{i'})}\| \leq \sum\limits_{i\in I} {c_2(\tau) \|t_i- f_{i'}\|}= O(h) \,.\end{equation}
			For the last step we note that $\|t_i - f_{i'}\| = \|t_i - f_i +f_i - f_{i'}\| \leq \|t_i - f_i\| + \|f_i- f_{i'}\| \leq O(h)  + (k+1)h$, due to the local approximation property and the distance constraint on the point $q_i$.
			Thus, the sum $\sum\limits_{i\in I} {A^{-1}_{i',i}t_i}$ is an affine combination of points $t_i$ on $T$ and therefore lies in $T$ as well (in a distance $\leq O(h)$), therefore it will follow that it is an $O(h^2)$ from the manifold.
			
			\item Next, similar considerations show that $\|r_i\| \leq \|f_i - f_{i'}\|^{2} \leq c_3((k+1)h+ O(h))^2 $. 
		\end{enumerate}
		To conclude, that based on items (a) and (b), the MLOP order of approximation to the manifold is $O(h^{2})$.
	\end{enumerate}
\end{proof} 
\vspace{-10mm}
\begin{figure}[H]
	\centering
	\label{fig:a1}\includegraphics[width=\textwidth,height=4.5cm,keepaspectratio]{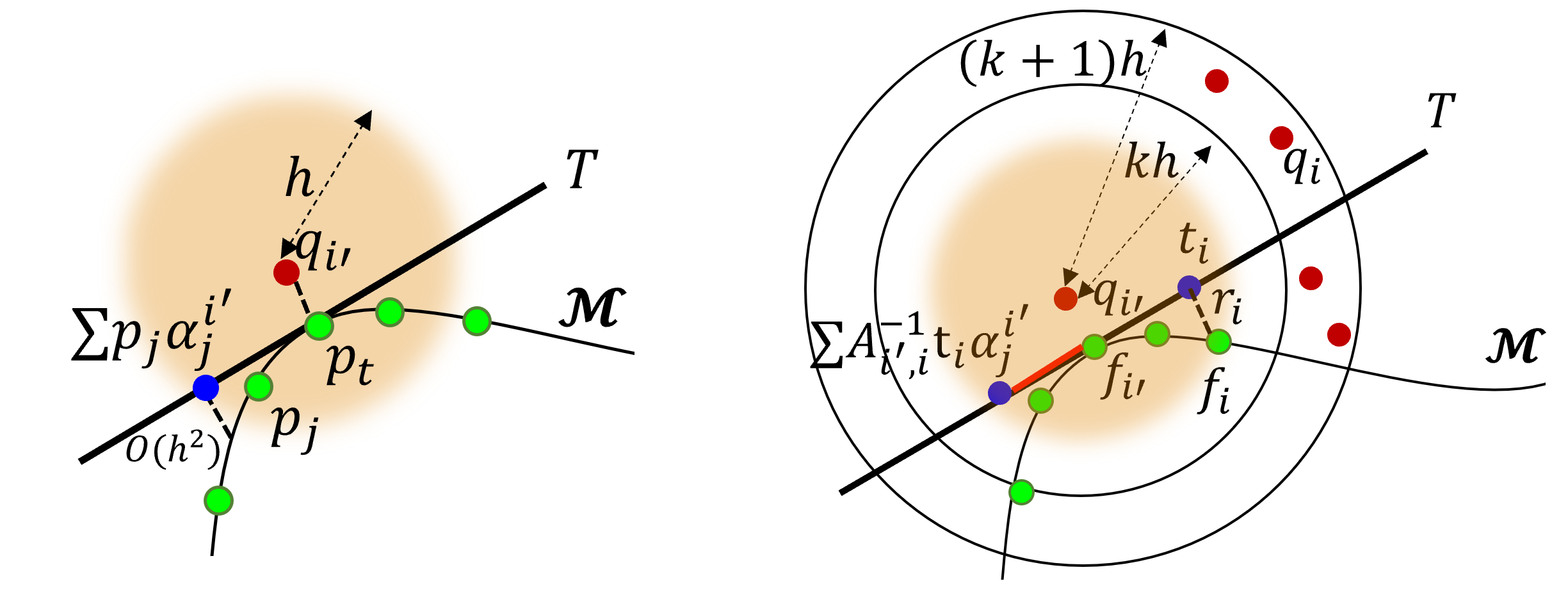}
	\caption{Illustration of the points participating in the estimate of the order of approximation. Left: demonstration why the affine combination of the $p_j$ points, in the neighborhood of $q_{i'}$, is of order $O(h^{2})$. Right: Illustration of the elements used in the estimation of the order of approximation. The $P$ points are marked in green, the $Q$ points in red, while the auxiliary points in the proof are marked in blue.}
	\label{fig:Approx_order}
\end{figure}

\subsection{Rate of Convergence}
\label{sec:convRate}
First, let us consider the gradient-descent rate of convergence of a Lipschitz-continuous strongly convex function. This rate of convergence depends on the condition number of the Hessian of the cost function, and so on the ratio between the smallest and the  largest eigenvalues of the Hessian,  i.e.,  $| 1- c \frac {\lambda_{\minmath}} {\lambda_{\maxmath}} |$, with $0<c<2$. Therefore, if our cost function would be convex, the rate of convergence could be $O(1-c/h^4)$. However, for non-convex optimization, the situation is much more complex. In our setting, where there is no convexity, one can analyze convergence to $\epsilon$-first-order stationary points, as defined below.
\begin{definition}\label{def:def6}
	A differentiable function $f(\cdot)$ is called $L$-smooth if for any $x_1, x_2$
	\begin{eqnarray*}  \|\nabla	 f(x_1) - \nabla f(x_2)\| \leq L\|x_1 - x_2\| \,.\end{eqnarray*}
\end{definition}
\begin{definition}\label{def:def7}
	If $f(\cdot)$ is a differentiable function, we say that $x$ is an $\epsilon$-first-order stationary point if $\|\nabla f(x)\| \leq \epsilon$\,.
\end{definition}
For the rate of convergence of our method, we will use the following theorem proved by Nestrove in \cite{nesterov2018lectures}.
\begin{Theorem}\label{thm:Nestrove} 
	Let $f(\cdot)$ be an $L$- smooth function that is bounded below. Then for any $\epsilon>0$, for the gradient descent with step size $\gamma = \frac{1}{\epsilon}$ and stop criterion  $\|\nabla f(x)\| \leq \epsilon$, the output will be an $\epsilon$-first-order stationary point, which will be reached after $k= \frac{L(f(x_0) – f^*)}{\epsilon^2}$ iterations. In case the starting point is close enough to the local minimum, the convergence is linear.
\end{Theorem}
It follows that in our case the rate of convergence is bounded.
\begin{Theorem}[Rate of convergence]\label{thm:convRate}
	Let the points-set $P=\{{p_j }\}_{j =1}^J$ be sampled near a $d$-dimensional manifold in $\mathbb{R}^n$ and let the assumptions in Theorem~\ref{thm:conv} be satisfied. Let the cost function $G$, defined as in (\ref{eq:1}), be an $L$-smooth function. For any $\epsilon >0$, let $Q^*$ be a local fixed-point solution of the gradient descent iterations, with step size $\gamma = \frac{1}{\epsilon}$. Set the termination condition as $\|\nabla G(Q)\| \leq \epsilon$. Then $Q^*$ is an $\epsilon$-first-order stationary point that will be reached after $k = \frac{L(G(Q^{(0)}) – G(Q^*))}{\epsilon^2}$ iterations, where $L=l^2$ and $l$ is given in \eqref{eq:l}.
\end{Theorem}
\begin{proof} 
	It is quite easy to verify that $G(Q)$ satisfies all the conditions of Theorem~\ref{thm:Nestrove}; in particular, the $L$-smoothness condition was proven above.
\end{proof} 
\begin{remark}
	In  our  case,  due  to  the  bound  on  $l$ in (\ref{eq:l}), we see that $k$ is of order $\frac{1}{h^{10}}$. However, in practice,  in our numerical examples, fewer iterations were needed to achieve convergence. In an example presented in the following section, with approximately \textup{800} noisy points $P$ and \textup{160} points in $Q$ (sampled in a certain area around a specific point), of a two-dimensional manifold embedded into a \textup{60}-dimensional space, the method converged in approximately \textup{500} iterations which took around \textup{90} seconds. When the initial set $Q$ was randomly sampled from $P$, we observed convergence in \textup{50} iterations which took 11 seconds.
\end{remark}
\begin{remark}It should be emphasized that the calculations of the gradient for each point are independent of one another, and in order to reduce the execution time, they can be run in multiple threads.
\end{remark}

\subsection{Uniqueness }
\label{sec:Uniqueness}
As shown in the previous section, convergence to a local minimum is guaranteed. However, since the cost function in \eqref{eq:1} is non-convex, a unique global solution can not be ensured. In order to address the uniqueness question, we have to rephrase the notion of uniqueness for our case. We do not refer to the uniqueness of the set $Q$, since there may be many sets $Q$ which satisfy the cost function \eqref{eq:1}, but to a common property of these optimal $Q$ sets, the fill-distance of their points. For instance, given a solution, its linear transformation can still minimize \eqref{eq:1}. This scenario is illustrated in Figure \ref{fig:uniqueness_v2}. In this example, which will be explained in detail in the experimental section, the orthogonal matrices in $\mathbb{R}^2$, which are represented by their angle, form a manifold. Although the two sets in  Figure \ref{fig:uniqueness_v2} (left and right) differ, they can still be solutions to the problem. 
\vspace{-5mm}
\begin{figure}[H]
	\centering
	\label{fig:gi}\includegraphics[width=\textwidth,height=4.0cm,keepaspectratio]{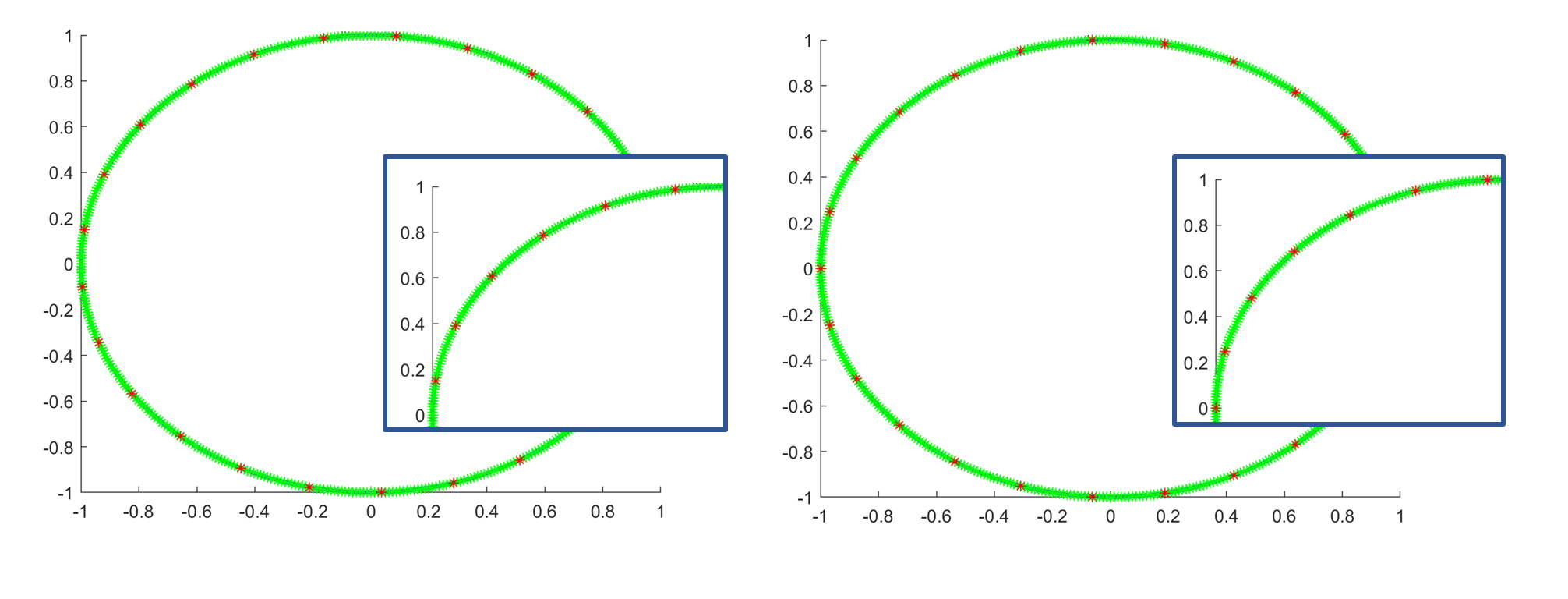}
	\caption{Manifold of orthogonal matrices: each matrix is represented by means of an angle (green), sampled with the same fill-distance, in two manners (red).}
	\label{fig:uniqueness_v2}
\end{figure}
\vspace{-5mm}
Thus the appropriate notation of uniqueness of the solution is as follows:
\begin{definition}\label{def:def_U1} 
	Let $Q_1$ and $Q_2$ be two point-sets uniformly sampled from a manifold $\mathcal{M}$, with fill-distance $h_2^1$ and $h_2^2$, respectively. Then $Q_1$ and $Q_2$ are said to be ``distribution equivalent'' if their fill-distances coincide ($h_2^1=h_2^2$). For a fixed fill-distance $h_q$, the corresponding class of distribution equivalent sets is denoted here by $[h_q]$.
\end{definition}

\begin{remark}\label{thm:uniqness} 
	Let $Q^*$ be a solution of the optimization problem \eqref{eq:1}, from points $P$. Then $Q^*$ is unique up to the equivalence class $[h_q]$. This follows from the definition of $h_q$, which specifies the number of $P$ points served by a single $q_i$, which uniquely define the equivalence class $[h_q]$ of the solution $Q^*$.
\end{remark}



\subsection{Complexity of the MLOP Algorithm}
\label{sec:complexity}

The complexity of the MLOP algorithm described in Algorithm \ref{alg:Alg1_} is based on a pre-step and a gradient decent iterations. As described in Section \ref{sec:highDimDistSec}, due to the curse of dimensionality and presence of noise all the norms are calculated in a lower dimension $m$. Thus, a pre-step to the MLOP algorithm is reducing the dimension of $P$ from $n$ to $m$ (where $m\ll n$), and have the complexity $nmJ$. In addition in every gradient descent step, and for every $q_i$ we reduce the dimension of current $Q$ which results in the complexity of $nmI$. As a result, a single gradient descent step is $O(I(nmI+I+J))$. With efficient neighboring calculation, this can be reduced to $O( I(nm\hat{I}+\hat{J})) $, where $\hat{I}$ and $\hat{J}$ are the numbers of points in the support of the weight function with respect to the $Q$ and $P$ sets, respectively (for instance, in the numerical examples below $\hat{J}$ was around 30 points, instead of 900 points in $P$). 
These operations are repeated $k$ times until convergence, where $k$ is bounded as in Theorem~\ref{thm:convRate}. Thus, the overall complexity is $O(nmJ + k I(nm\hat{I}+\hat{J}))$.

\begin{corollary}
	\label{complex_MLOP}
	
	Given a point-set $P=\{{p_j }\}_{j =1}^J$ sampled near a $d$-dimensional manifold $\mathcal{M} \in \mathbb{R}^n$,  let $Q=\{q_i \}_{i=1}^I$ be a set of points that will provide the desired manifold reconstruction. Then the complexity of the {\rm MLOP} algorithm is $O(nmJ + k I(nm\widehat{I}+\widehat{J}))$, where the number of iterations $k$ is bounded as in Theorem~{\rm \ref{thm:convRate}}, $m \ll n$ is the smaller dimension to which we reduce the dimension of the data, and $\widehat{I}$ and $\widehat{J}$ are the numbers of points in the support of the weight functions $\hat w_{i,i'}$, $w_{i,j}$ with the $Q$-set and $P$-set, respectively. Thus, the approximation is linear in the ambient dimension $n$, and does not depend on the intrinsic dimension $d$.
		
\end{corollary}

\section{Numerical Examples}
\label{sec:NumericalExamples1}
In this section, we present some numerical examples which demonstrate the validity of our method, as well as its robustness under different scenarios, for example, diverse manifold topologies, different amounts of noise, and many intrinsic dimensions. In all the examples the input points $P$ were sampled uniformly in the parameter space. Next, a uniform noise $U(-\sigma, \sigma)$ with magnitude $\sigma$ was added. Then the set $Q$ was initialized by sampling from the set $P$ around a certain selected point. In what follows we illustrate the results of applying the MLOP algorithm.

\subsection*{One-Dimensional Orthogonal Matrices}
\label{sec:O2_example}
Consider the case of the manifold $O(2)$ of orthogonal matrices, embedded into a 60-dimensional linear space by using the parameterization
\begin{eqnarray*}\hat p=[\cos (\theta), -\sin(\theta), \sin(\theta), \cos(\theta), 0, \dotsc, 0]\,,\end{eqnarray*} 
where $\theta \in [-\pi, \pi]$. The input data $\hat P$ were constructed by sampling 500 equally distributed points in the parameter space. Next, we randomly sampled an orthogonal matrix $A \in \R^{60 \times 60}$, and created a new point-set via non-trivial vector embedding
\begin{equation} 
P = A \hat P \,.
\end{equation}
Later we added a uniform noise $U(-0.2, 0.2)$, and initialized the set $Q$ selecting $50$ points around a certain point. Figure \ref{fig:O2_img} left illustrates the first two coordinates of the points in our set (after a multiplication with $A^{-1}$). The noisy sampled points are shown in green, while the initial reconstruction points are shown in red. Figure \ref{fig:O2_img} right shows the reconstructed and denoised manifold of orthogonal matrices, after $500$ iterations of the MLOP algorithm (red).
\vspace{-12mm}
\begin{figure}[H]
	\centering
	\label{fig:k}\includegraphics[width=\textwidth,height=5.5cm,keepaspectratio]{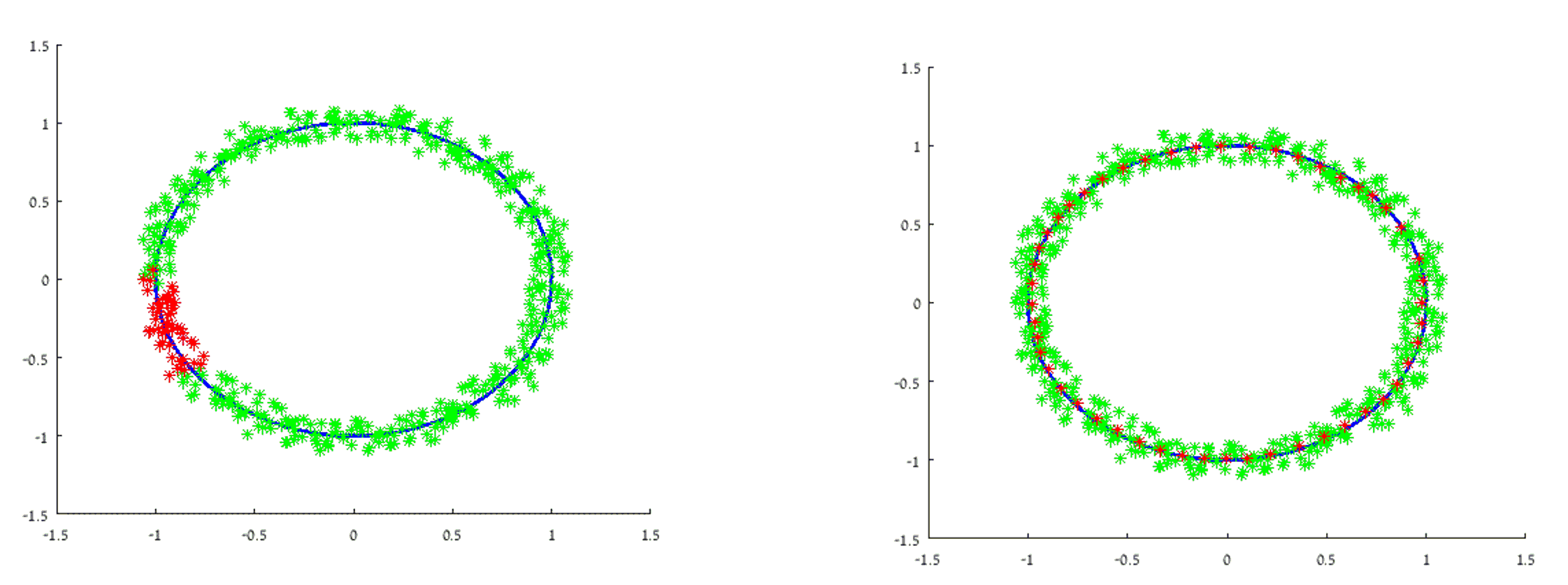}
	\caption{Manifold of orthogonal matrices embedded into a 60-dimensional space. Shown are the first two coordinates of the point-set (after multiplication with $A^{-1}$). Left: Scattered data with uniformly distributed noise $U(-0.2; 0.2)$ (green), and the initial point-set $Q^{(0)}$ (red) Right: The resulting point-set of MLOP algorithm after 500 iterations, $Q^{(500)}$ (red) overlaying the noisy samples (green).}
	\label{fig:O2_img}
\end{figure}
\vspace{-3mm}
\subsection*{Three-Dimensional Cone Structure}
\label{sec:Cone_example}
Next, we demonstrate the ability of the MLOP to cope with a geometric structure of different dimensions at different locations. Here we combined a 3-dimensional manifold, namely, a cone structure, with a one-dimensional manifold, namely, a line segment. This object was embedded into a 60-dimensional linear space. The cone’s parameterization used was
\begin{eqnarray*} p=t v_1+\frac {e^{-{R^2}}}{\sqrt 2}(\cos(u)v_2+ \sin(u)v_3)\,, \end{eqnarray*}  
where $v_1=[1,1,1,1,0, \dotsc, 0], v_2=[0, 1, -1, 0, 0,\dotsc,0], v_3=[1, 0, 0, -1, 0,\dotsc,0]$, $(v_1,v_2,v_3) \in \R^{60}$, $t\in [0,2]$, $R\in [0,2.5]$, and $u \in [0.1\pi,1.5\pi]$. We sampled 720 points from the structure with added uniformly distributed noise of magnitude $0.2$. The initial set $Q^{(0)}$ of size $144$ was selected (Figure \ref{fig:cone_img}  left), and $500$ iterations of the MLOP were performed to reconstruct and denoise the geometrical structure (Figure \ref{fig:cone_img} right).
\vspace{-12mm}
\begin{figure}[H]
	\centering
	\label{fig:j}\includegraphics[width=\textwidth,height=5.5cm,keepaspectratio]{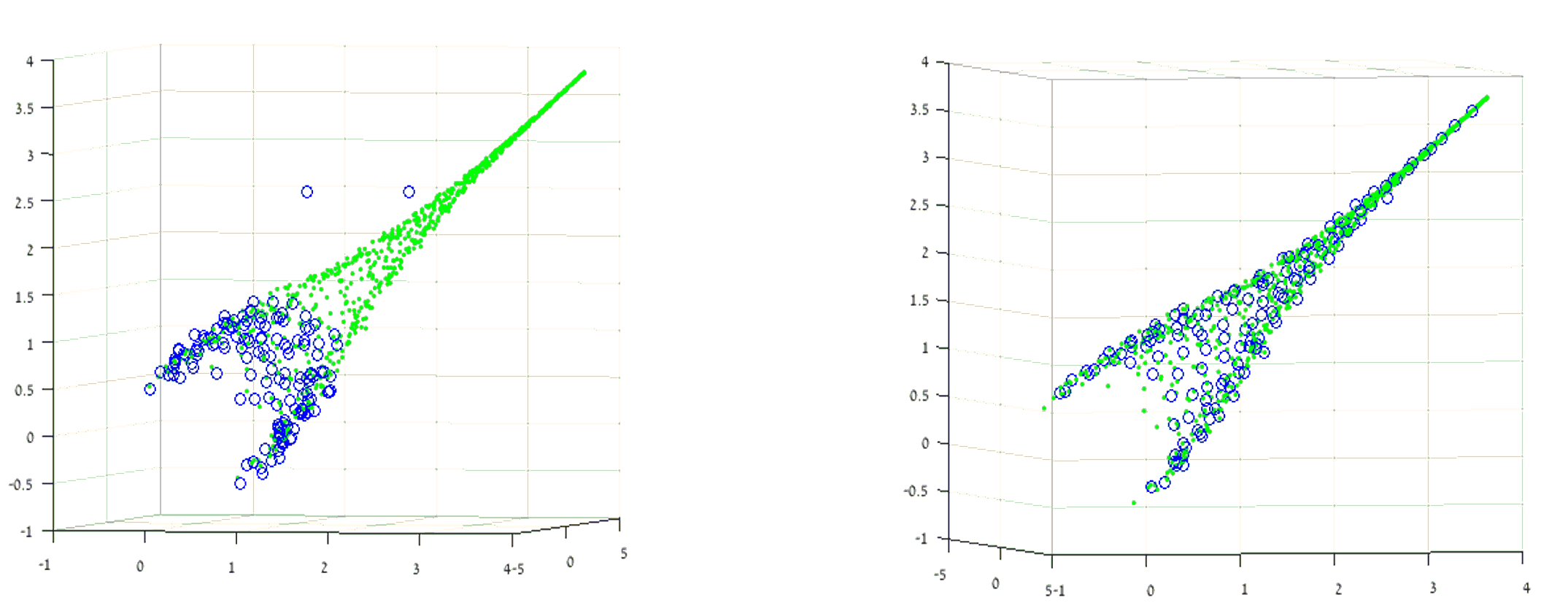}
	\caption{Geometrical structure of changing dimension. Combination of a cone and a line segment, embedded into a 60-dimensional space. The first three coordinates of the point-set are shown. Left: Scattered data with uniformly distributed noise $U(-0.2; 0.2)$ (green), and the initial point-set $Q^{(0)}$ (blue) Right: The point-set generated by the MLOP algorithm after 500 iterations, $Q^{(500)}$ (blue) overlaying the noisy samples (green).} 
	\label{fig:cone_img}
\end{figure}
\subsection*{Two-Dimensional Cylindrical Structure}
\label{sec:Cylindr_example}
In the next example, we embedded a two-dimensional cylindrical structure into a 60-dimensional linear space. We sampled the structure using the parameterization
\begin{eqnarray*}p=t v_1+\frac{R}{\sqrt 2}(\cos(u)v_2+  \sin(u)v_3)\,,\end{eqnarray*}
where $v_1=[1,1,1,1,1, \dotsc ,1]$, $v_2=[0,1,-1,0,0, \dotsc ,0], v_3=[1,0,0,-1,0, \dotsc ,0]$,  
$(v_1,v_2,v_3 \in \R^{60})$, $t \in [0,2]$ and $u \in [0.1 \pi,1.5 \pi]$. Using this representation 816 equally distributed (in parameter space) points were sampled with uniformly distributed noise (i.e., $U(-0.1, 0.1)$). As can be seen in Figure \ref{fig:cylinder_img} left, the initial set $Q^{(0)}$ of size 163 was selected very roughly, and 500 iterations of the MLOP were performed to reconstruct the cylindrical structure, shown in Figure \ref{fig:cylinder_img} right.
\vspace{-10mm}
\begin{figure}[H]
	\centering
	\label{fig:h}\includegraphics[width=\textwidth,height=5.5cm,keepaspectratio]{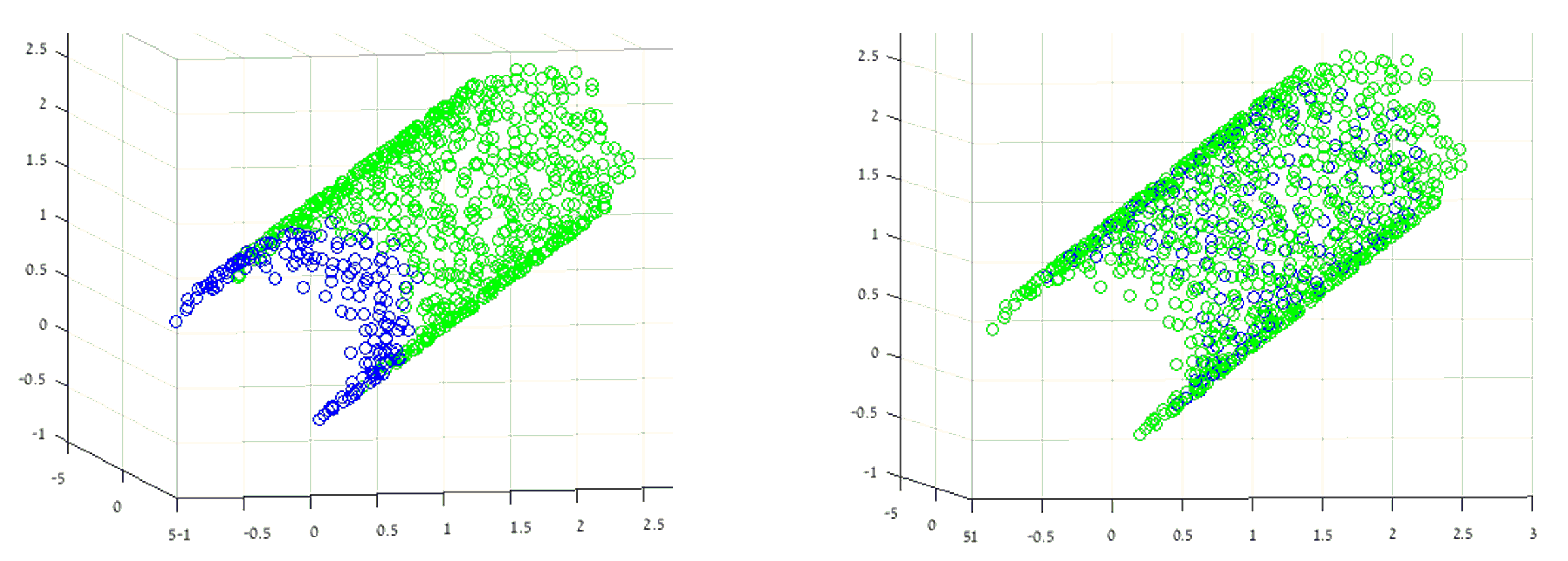}
	\caption{Cylindrical structure embedded into a 60-dimensional space. The first three coordinates of the point-set are shown. Left: Scattered data with uniformly distributed noise $U(-0.1; 0.1)$ (green), and the initial point-set $Q^{(0)}$ (blue) Right: The point-set generated by the MLOP algorithm after 500 iterations, $Q^{(500)}$ (blue) overlaying the noisy samples (green).} 
	\label{fig:cylinder_img}
\end{figure} 
\subsection*{Robustness to Noise}
The noise level has a direct influence on the accuracy of the reconstruction. Here we examine the robustness of the MLOP under various levels of noise. Our test was performed on the two-dimensional cylindrical structure embedded into 60-dimensions, with various amounts on noise magnitude ($0$, $0.1$, $0.2$, and $0.5$). The accuracy was calculated as the relative error of the reconstruction $Q$, against a densely sampled noise-free cylindrical structure. The norm used for accuracy calculations was the one that is based on linear sketching, as defined in Section \ref{sec:highDimDistSec}. As can be seen in Figure \ref{fig:cylinder_noise_vs_error}, even with a noise level of 0.5, the reconstruction quality is satisfactory (with a relative error of $0.15$).
\vspace{-10mm}
\begin{figure}[H]
	\centering
	\label{fig:i}\includegraphics[width=\textwidth,height=5cm,keepaspectratio]{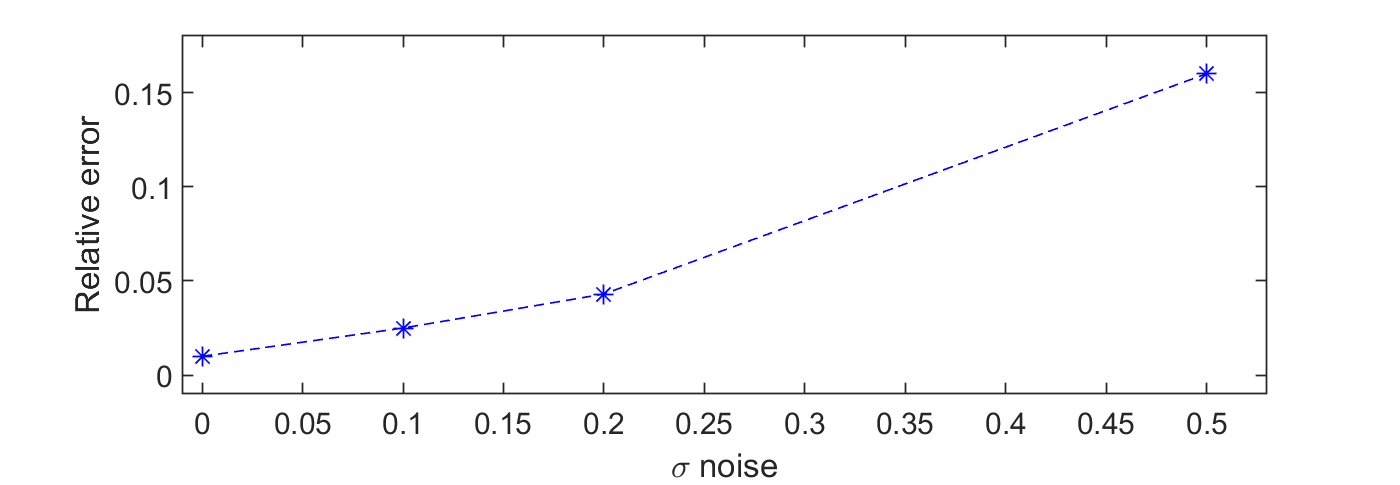}
	\caption{Effect of noise level on the reconstruction accuracy of a cylindrical structure embedded into a 60-dimensional space.}
	\label{fig:cylinder_noise_vs_error}
\end{figure}

\subsection*{Six-dimensional cylindrical structure}
\label{sec:10dimCylinder_example_1}
Finally, we tested our method on manifolds of the higher dimension by utilizing an $n$-sphere to generate an $(n+1)$-dimensional cylinder (in the example of the two-dimensional cylinder, we used a circle to generate the structure). Here, we utilized a five-dimensional sphere to build a six-dimensional manifold, using the parameterization
\begin{align*}
x_1 = R \cos(u_1) \,,\quad
x_2 = R \sin(u_1)\cos(u_2),\quad 
\ldots,\quad 
x_{6} = R \sin(u_1)\sin(u_2)\cdots\sin(u_5)\sin(u_6)\,. 
\end{align*}
We then embedded the sampled data in a 60-dimensional space
\begin{eqnarray}p=t v_0 + R^2 [x_1,x_2, x_3, x_4, x_5, x_6, 0, \dotsc ,0]\,,\end{eqnarray}
where $R = 1.5$, $t \in [0,2]$, $u_i \in [0.1 \pi, 0.6 \pi]$, and $v_0 \in \R^{60}$ is a vector with 1's in positions $1,...,d+1$ and 0 in the remaining positions. In this test, we sampled 1200 points from this manifold and added a noise $U(-0.1, 0.1)$. The initial reconstruction set was chosen to consist of randomly selected 460 points. The method converged after approximately 300 iterations. To avoid trying to visualize a six-dimensional manifold, we plot in Figure \ref{fig:sixD_cylinder_noise_vs_error} the cross-section of the cylindrical structure in three-dimensions. We evaluate the efficiency of the denoising effect by calculating the maximum relative error, root mean square error, and variance of both the initial $Q^{(0)}$ points and the noise-free reconstruction set $Q^{(300)}$ with respect to the closest point in the clean reference data. As a result, the errors if $Q^{(0)}$ are $0.083$, $0.32 \pm 0.0007$, and of the noise-free reconstruction are $0.058$, $0.28 \pm 0.0006$. Thus, we see that in this scenario of non-trivial intrinsic dimension of the manifold the error decrease dramatically. In addition, the fill-distance of the initial random $Q^{(0)}$ set was $0.36$, and $0.32$ in the reconstruction. Thus, we also observe the effect of quasi-uniform sampling after applying the MLOP.

\begin{figure}[H]
	\centering
	\label{fig:i}\includegraphics[width=\textwidth,height=5.5cm,keepaspectratio]{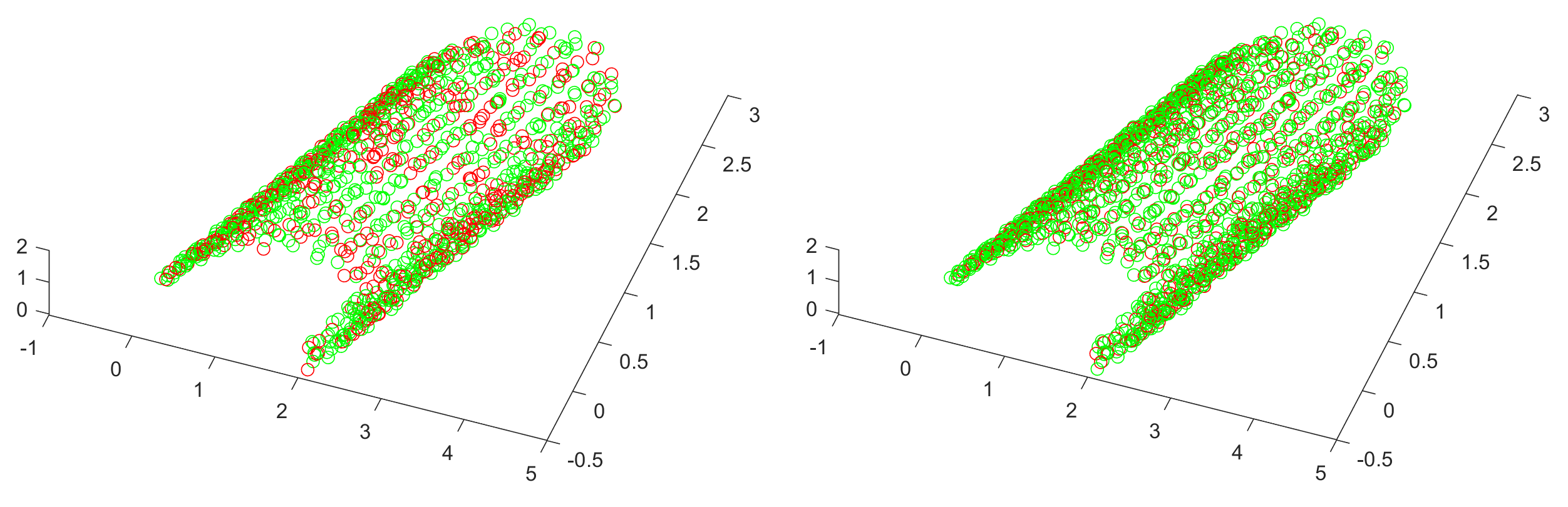}
	\caption{Six-dimensional cylindrical structure embedded in a 60-dimensional space. The cross-section of the six-dimensional cylindrical structure is plotted in three-dimensions. Left: Scattered data with uniformly distributed noise $U(-0.1; 0.1)$ (green), and the initial point-set $Q^{(0)}$ (red) Right: The point-set generated by the MLOP algorithm after 300 iterations, $Q^{(300)}$ (red) overlaying the noisy samples (green).}
	\label{fig:sixD_cylinder_noise_vs_error}
\end{figure}

\subsection*{Applications to Image Processing}
\label{sec:elipses_example}
Manifold denoising and reconstruction methodology can be also applied to image processing problems. At the beginning of this paper, we described the cryo-EM (in Figure \ref{fig:cryoFig}) which motivated our study. In this framework a manifold is created by acquiring images of a single object in various directions. As a preliminary example, before addressing the real case of cryo-EM, we simulated data that resemble the cryo-EM conditions. Specifically, we sampled $900$ images of ellipses of size $20 \times 20$. The ellipses were centered and no rotations were used. Thus, we have $900$ samples of a 2-dimensional submanifold embedded in $\R^{400}$. We added a Gaussian noise $N(0; 0.05)$ to each pixel. Figure \ref{fig:eliipses_1_img} shows the sample of the manifold (with some zoom-in examples), along with a graph where the $(x,y)$ - coordinates of each point are the ellipse radii. For the execution of the MLOP, we took $180$ ellipses as the initial sample points (Figure \ref{fig:eliipses_2_img} left). As can be seen in Figure \ref{fig:eliipses_2_img} right, after 1000 iterations the samples were cleaned, while the radii distribution graph shows that the radii domain is fully sampled.

We evaluated the MLOP denoise performance on the ellipses samples  $Q$. We measured the SNR as $SNR = \frac{\mu}{\sigma}$ on the background pixels of each ellipse image (where $\mu$ is the average signal value, while $\sigma$ is the standard deviation). We observe that the median SNR of the set $Q$ increased after applying the MLOP denoising, from $15.6$ to $36.5$. This gives us a quantitative measure of the denoising performed by the MLOP (as can also be seen in Figure \ref{fig:eliipses_2_img} in the zoomed-in areas).

\vspace{-10mm}
\begin{figure}[H]
	\centering
	\label{fig:l}\includegraphics[width=\textwidth,height=7cm,keepaspectratio]{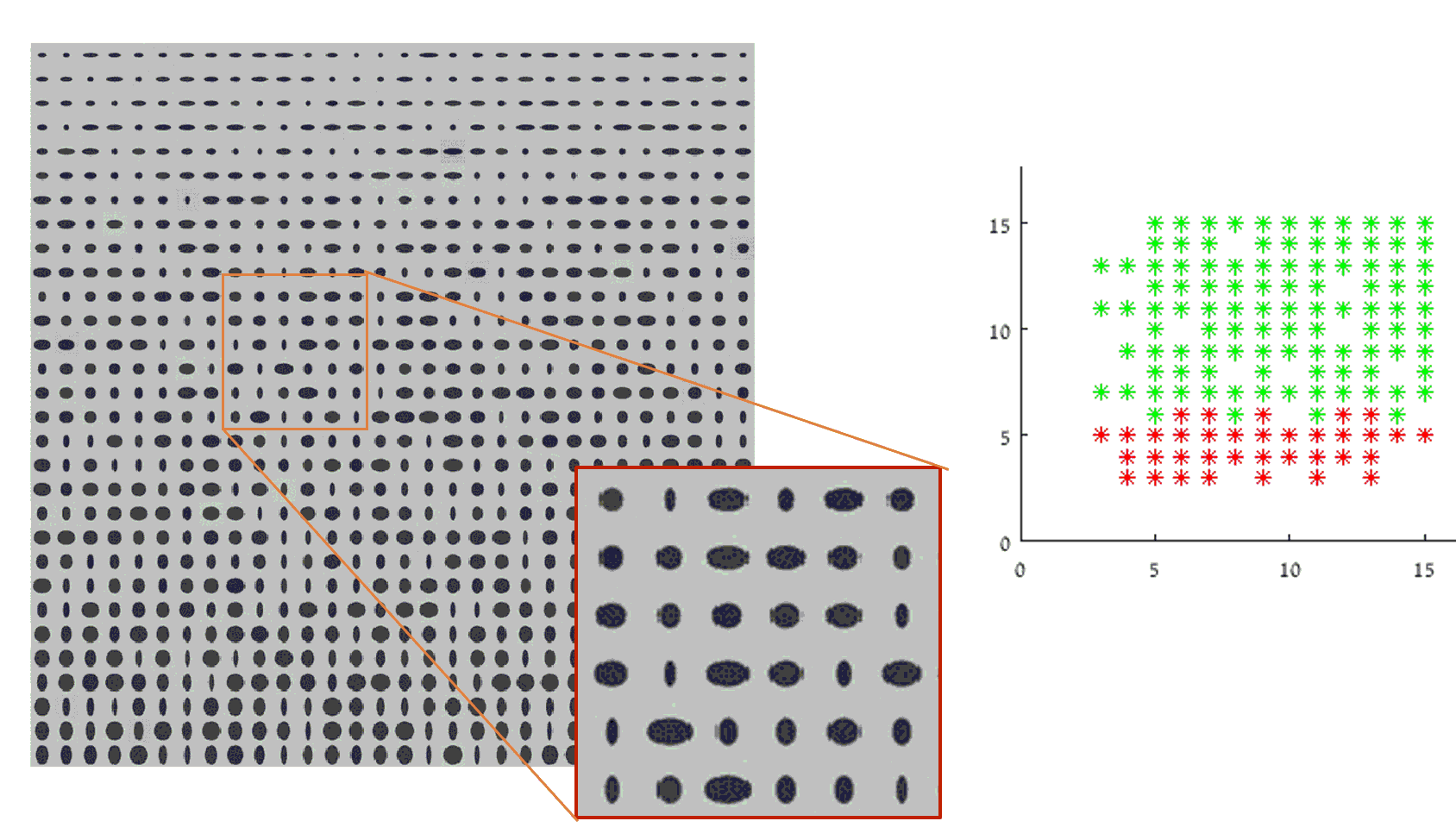}
	\caption{Left: Images of ellipses with varying radii that were sampled from a 2-dimensional manifold, prior to adding noise, which will form the $P$ set. Right: a graph depicting the radii of the ellipses, with the coordinates of points given by these radii. The manifold samples are shown in green ($P$), while the initial set $Q^{(0)}$ is shown in red.}
	\label{fig:eliipses_1_img}
\end{figure}
\vspace{-14mm}
\begin{figure}[H]
	\centering
	\label{fig:o}\includegraphics[width=\textwidth,height=7cm,keepaspectratio]{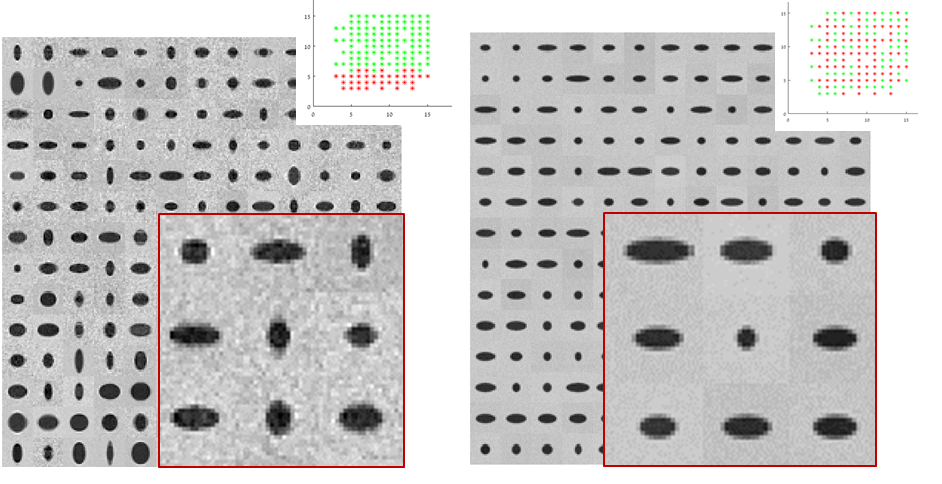}
	\caption{The samples that were used to reconstruct the manifold. Each side of the figure consists of an image of the samples, a zoomed-in area, and a graph of sample radii. The manifold samples are shown in green, while the initial set is shown in red. Left: the initial configuration of points sampled from the 2-dimensional manifold. Right: the manifold reconstruction configuration after 1000 iterations.}
	\label{fig:eliipses_2_img}
\end{figure}

\section{MLOP Denoise Benefits}
The current section dealt both with manifold reconstruction and cleaning of high amounts of noise. The denoising property was induced by the first term in \eqref{eq:1}, which performs smoothing of $p_j$ samples in the neighborhood of the examined point $q_i$ . This term is inspired by the $L_1$-median \cite{vardi2000multivariate}, and thus is robust to high amounts of noise. This fact was demonstrated in the "Robustness to Noise" subsection in \cite{faigenbaum2016algorithmic}, where the effect of various levels of noise on reconstruction accuracy was examined. The test demonstrated the robustness of the MLOP method to various amounts of noise magnitude ($0$, $0.1$, $0.2$, and $0.5$), on a two-dimensional cylindrical structure embedded into 60-dimensions. The calculation of relative error of the reconstruction $Q$, against a densely sampled noise-free cylindrical structure, showed good results even at a noise level of 0.5 (with a relative error of $0.15$). Thus, it is natural to use MLOP as a pre-processing step prior to performing mining tasks on the data. 

In this section, we demonstrate the effectiveness of high-dimensional denoising in the case of local PCA. In our test we examine a set of points $X=\{x_i\}$, with a fill-distance $h$. We calculate PCA for each point $x_i$ using its neighboring points ${x_j}$, which maintain the constraint  $\|x_i-x_j\| < h$. Next, we extract the first eigenvector and evaluate its accuracy with respect to the first eigenvector of a PCA executed on clean reference data. Specifically, for each point $x_i$ we find the closest point in the clean reference data and calculate the cosine distance between the corresponding PCA first eigenvectors (the error is given in degrees). Next, we determine the median of the errors stemming from all the points $X$. It is important to note that the error is tightly connected with the number of points in the set, with their fill-distance, and naturally with the noise levels. For example, on clean data with 160 points randomly sampled from a manifold, the error was 11.8, while with 7000 points, the error decrease to 0.2. This stems from the fact that taking a larger number of points in the neighbor of a point $x_i$ leads to a more accurate eigenvector. This fact has to be taken into account in error analysis.

The numerical calculations were performed on the example of a two-dimensional cylindrical structure embedded into a 60-dimensional linear space. We sampled the structure using the parameterization
\begin{eqnarray*}p=t v_1+\frac{R}{\sqrt 2}(\cos(u)v_2+  \sin(u)v_3)\,,\end{eqnarray*}
where $v_1=[1,1,1,1,1, \dotsc ,1]$, $v_2=[0,1,-1,0,0, \dotsc ,0], v_3=[1,0,0,-1,0, \dotsc ,0]$  
$(v_1,v_2,v_3 \in \R^{60})$, $t\in [0,2]$ and $u \in [0.1 \pi,1.5 \pi]$. Using this representation, 816 uniformly distributed (in parameter space) points were sampled with uniformly distributed noise (i.e., $U(-0.2, 0.2)$). As can be seen in Figure \ref{fig:pca1} left, after 500 iterations of the MLOP algorithm, the cylindrical structure was reconstructed with high accuracy (red points).

The experiments testing the efficiency of MLOP denoising were carried out on five data sets, all of size 160:
\begin{enumerate}[noitemsep]
	\item Noise-free data. 
	\item Noise data with additive noise of 0.1. 
	\item Data denoised by the MLOP from the data in item 2.
	\item Noisy data with additive noise of 0.2. 
	\item Data denoised by the MLOP from the data in item 4. 
\end{enumerate}	

The results for noise levels of 0.1 and 0.2 are presented in Figure \ref{fig:pca1} right. To achieve a robust error value, we performed ten bootstrap iterations for the "noise-free", as well as "noisy data" data-sets, where we randomly sampled the manifold, and calculated the median PCA error of the iterations. As expected, the effect of the MLOP denoising is to improve the accuracy of the local PCA calculations. One can see that the noise level has a small effect on the error (increasing it from 7.9 to 8.2, for the 0.1 and 0.2 noise level respectively). An additional benefit is that the accuracy of the denoised data is superior the one of the is noise-free data. The reason for this is the quasi-uniform manifold sampling which MLOP carries out accordingly due to the second term in \eqref{eq:1}, while the noise-free samples come from randomly sampled points (which not necessarily sample the manifold uniformly).

\vspace{-10mm}
\begin{figure}[H]
	\centering
	\label{fig:aa1}\includegraphics[width=\textwidth,height=6cm,keepaspectratio]{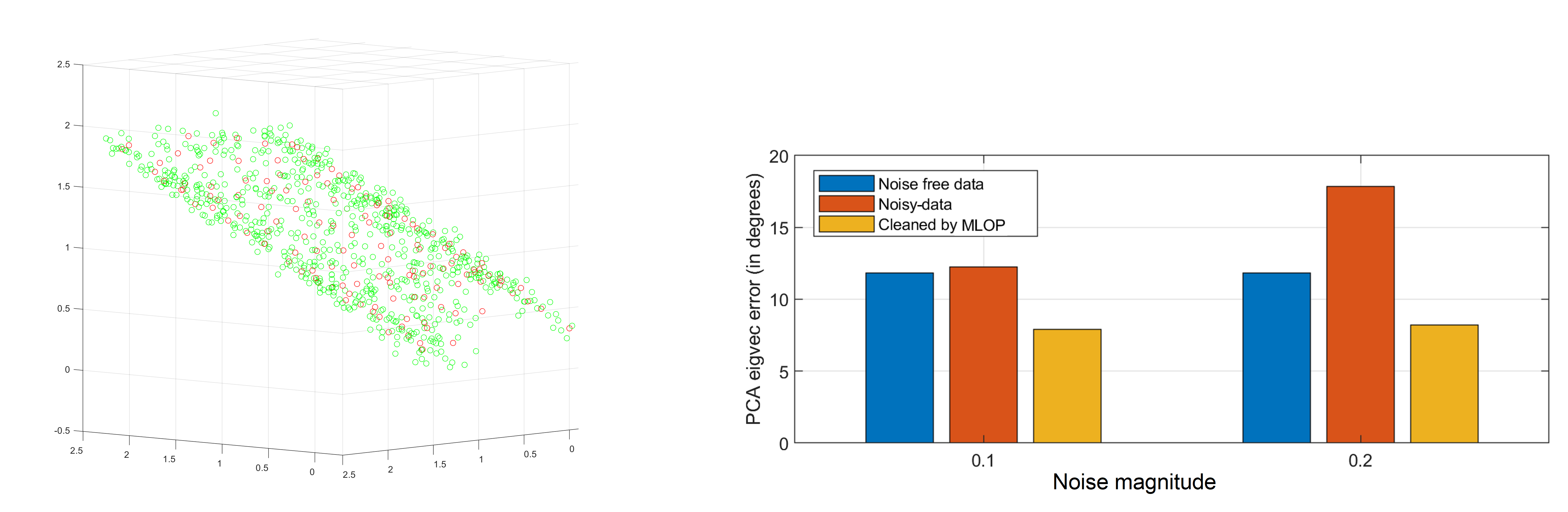}
	\caption{Left: Cylindrical structure, sampled with noise $U(-0.2,0.2)$, and embedded in $\R^{60}$. The figure presents the first three coordinates of the points set. The point-set generated by the MLOP algorithm after 500 iterations, $Q^{(500)}$ (red) overlaying the noisy samples (green). Right: illustration of the MLOP denoising effect on the accuracy of PCA calculations. The graphs present the error of the first eigenvector of local PCA calculated on noise-free, noisy, and denoised data.}
	\label{fig:pca1}
\end{figure}

%
%
%
%


\section{Discussion and Future Directions}
\label{sec:conclusions}

The big-data era gave rise to many challenges related to processing, analyzing, and understanding high-dimensional data. Among these challenges are the presence of noise, outliers, incomplete data, or insufficient data. In this paper, we introduced a framework that can address these issues, raised by high-dimensional data, in an efficient and robust manner. We propose a method for manifold reconstruction and denoising in high-dimensional space. Over the years, several solutions were suggested to cope with the reconstruction problem in high-dimensional space. However, they have a hard time handling noisy data, non-uniformly sampled, with no assumption on the data. As a result, manifold reconstruction in noisy conditions in high-dimensional space is still an open question. In our research, we address the manifold approximation question by extending the LOP \cite{lipman2007parameterization} algorithm to the high-dimensional case. We develop a new algorithm, called Manifold Locally Optimal Projection (MLOP). We look for a noise-free manifold reconstruction in high-dimensional space by solving a non-convex optimization problem which leverages L1-median generalization to high dimension, while requiring a quasi-uniform distribution of points in the reconstruction. We prove that the MLOP method converges to a local stationary solution with a bounded linear rate of convergence when the starting point is close enough to the local minimum. In addition, we showed that the manifold order of approximation is $O(h^2)$, where $h$ is the representative distance between the points, and the complexity is linear in the ambient dimension and does not depend on the intrinsic dimension.

The numerical examples demonstrate the applicability of the proposed method to various high-dimensional scenarios. This opens the door to different applications. First, it is possible to extend the methodology for approximating function on a manifold in noisy conditions (both in the function domain and in its codomain). Next, it is possible to enhance the MLOP to address the manifold repairing in the high-dimension problem, where input data have holes, and the target is to find a noise-free reconstruction of the manifold that will amend the holes and complete the missing information. Then, using the MLOP methodology it is possible to address the problem of multivariate k-L1-medians in high-dimensional cases. This can be achieved by finding the service centers by using the MLOP out-of-the-box. Last, but not least, the flexibility of selecting the amounts of points in the reconstruction and set the density paves the way for manifold upsampling and downsampling, and for manifold compression. Thus, we see the MLOP framework is a cornerstone method for handling high-dimensional noisy data.

\section*{Acknowledgments}
We would like to thank Dr. Barak Sober for valuable discussions, and comments. This study was supported by a generous donation from Mr. Jacques Chahine, made through the French Friends of Tel Aviv University, and was partially supported by ISF grant 2062/18.

\bibliographystyle{spmpsci}      
\bibliography{references_MLOP}

%
%

\end{document}